\documentclass[12pt]{amsart}
\usepackage{graphicx,amsthm,amssymb,amsmath,setspace,ulsy,tikz-cd,amsfonts,mathrsfs,accents,fancyhdr,nameref,multicol,wasysym,subcaption} 
\usepackage[all]{xy}
\usepackage[shortlabels]{enumitem}
\usepackage[square,numbers]{natbib}

\usetikzlibrary{decorations.markings,arrows}

\title{Multistring based matrices}
\author{David Freund}

\newtheorem{thm}{Theorem}[section]
\newtheorem{lem}[thm]{Lemma}

\newtheorem{prop}[thm]{Proposition}

\theoremstyle{remark}
\newtheorem{rem}[thm]{Remark}
\newtheorem{ex}[thm]{Example}

\newcommand{\R}{\ensuremath{\mathbb R}}

\newcommand{\Z}{\ensuremath{\mathbb Z}}
\newcommand{\es}{\varnothing} 
\newcommand{\A}{\alpha}
\newcommand{\B}{\beta}
\newcommand{\p}{\varphi}
\newcommand{\E}{\varepsilon}

\newcommand{\Gam}{\Gamma}

\newcommand{\Id}{\operatorname{Id}}

\newcommand{\sign}{\operatorname{sign}}

\newcommand{\arr}{\operatorname{arr}} 

\newcommand{\abs}[1]{\left|{#1}\right|} 
\newcommand{\set}[1]{\left\{#1\right\}} 
\newcommand{\paren}[1]{\left(#1\right)} 
\newcommand{\flr}[1]{\left\lfloor #1\right\rfloor} 
\newcommand{\ov}[1]{\overline{#1}} 
\newcommand{\ul}[1]{\underline{#1}} 
\newcommand{\ds}[1]{\displaystyle{#1}} 

\newcommand{\CRMoves}[1]{\begin{tikzpicture}[#1,thick]
	\draw[decoration={markings,mark=at position 0.75 with {\arrow[scale=1.5]{>}}},
        postaction={decorate},opacity=.3,xshift=-2cm] (0,0) circle (1);
	\draw[xshift=-2cm,-] (-75:1) arc [radius=1,start angle=-75,end angle=55]; 
	
	\draw[<->] (-.5,.5) -- node[pos=.5,above] {1a} (.5,.875); 
	\draw[<->] (-.5,-.5) -- node[pos=.5,below] {1b} (.5,-.875);

	\draw[decoration={markings,mark=at position 0.75 with {\arrow[scale=1.5]{>}}},
        postaction={decorate},opacity=.3,xshift=2cm,yshift=1.25cm] (0,0) circle (1);
	\draw[xshift=2cm,yshift=1.25cm,-] (-75:1) arc [radius=1,start angle=-75,end angle=55]; 
	\draw[xshift=2cm,->,yshift=1.25cm] (-50:1) to[out=130,in=210] (30:1);
	
	\draw[decoration={markings,mark=at position 0.75 with {\arrow[scale=1.5]{>}}},
        postaction={decorate},opacity=.3,xshift=2cm,yshift=-1.25cm] (0,0) circle (1);
	\draw[xshift=2cm,-,yshift=-1.25cm] (-75:1) arc [radius=1,start angle=-75,end angle=55];
	\draw[xshift=2cm,<-,yshift=-1.25cm] (-50:1) to[out=130,in=210] (30:1);
	
	\begin{scope}[xshift=6cm]
	
	\draw[decoration={markings,mark=at position 0.75 with {\arrow[scale=1.5]{>}}},
        postaction={decorate},opacity=.3] (0,0) circle (1);
	\draw[-] (145:1) arc [radius=1,start angle=145,end angle=215];
	\draw[-] (35:1) arc [radius=1,start angle=35,end angle=-35];
	
	\draw[<->,xshift=2cm] (-.5,0) -- node[auto] {2} (.5,0);
	\end{scope}

	\begin{scope}[xshift=10cm]
	
	\draw[decoration={markings,mark=at position 0.75 with {\arrow[scale=1.5]{>}}},
        postaction={decorate},opacity=.3] (0,0) circle (1);
	\draw[-] (145:1) arc [radius=1,start angle=145,end angle=215];
	\draw[-] (35:1) arc [radius=1,start angle=35,end angle=-35];
		
	\draw[<-] (170:1) to (10:1);
	\draw[->] (190:1) to (-10:1);
	\end{scope}

	\begin{scope}[yshift=-4.935cm,xshift=0cm] 
	
	\draw[decoration={markings,mark=at position 0.5 with {\arrow[scale=1.5]{>}}},
        postaction={decorate},opacity=.3,xshift=-2cm] (0,0) circle (1);
	\draw[xshift=-2cm,-] (115:1) arc [radius=1,start angle=115,end angle=155]; 
	\draw[xshift=-2cm,-] (250:1) arc [radius=1,start angle=250,end angle=290]; 
	\draw[xshift=-2cm,-] (25:1) arc [radius=1,start angle=25,end angle=65]; 
	
	\draw[xshift=-2cm,<-] (125:1) to[out=-50,in=230] (55:1); 
	\draw[xshift=-2cm,->] (145:1) to[out=-40,in=85] (260:1); 
	\draw[xshift=-2cm,->] (280:1) to[out=95,in=220] (35:1); 

	\draw[<->] (-.5,0) -- node[auto] {3a} (.5,0);
	
	\draw[decoration={markings,mark=at position 0.5 with {\arrow[scale=1.5]{>}}},
        postaction={decorate},opacity=.3,xshift=2cm] (0,0) circle (1);
	\draw[-,xshift=2cm] (115:1) arc [radius=1,start angle=115,end angle=155]; 
	\draw[-,xshift=2cm] (250:1) arc [radius=1,start angle=250,end angle=290]; 
	\draw[-,xshift=2cm] (25:1) arc [radius=1,start angle=25,end angle=65]; 
	
	\draw[xshift=2cm,<-] (145:1) to[out=-35,in=215] (35:1); 
	\draw[xshift=2cm,->] (125:1) to[out=-55,in=100] (280:1); 
	\draw[xshift=2cm,->] (260:1) to[out=80,in=225] (55:1); 
	\end{scope}
	
	\begin{scope}[yshift=-4.935cm,xshift=8cm] 
	\draw[decoration={markings,mark=at position 0.5 with {\arrow[scale=1.5]{>}}},
        postaction={decorate},opacity=.3,xshift=-2cm] (0,0) circle (1);
	\draw[-,xshift=-2cm] (115:1) arc [radius=1,start angle=115,end angle=155]; 
	\draw[-,xshift=-2cm] (250:1) arc [radius=1,start angle=250,end angle=290]; 
	\draw[-,xshift=-2cm] (25:1) arc [radius=1,start angle=25,end angle=65]; 
	
	\draw[xshift=-2cm,->] (125:1) to[out=-55,in=80] (260:1); 
	\draw[xshift=-2cm,->] (280:1) to[out=100,in=225] (55:1); 
	\draw[xshift=-2cm,->] (145:1) to[out=-35,in=215] (35:1); 

	\draw[<->] (-.5,0) -- node[auto] {3b} (.5,0);
	
	\draw[decoration={markings,mark=at position 0.5 with {\arrow[scale=1.5]{>}}},
        postaction={decorate},opacity=.3,xshift=2cm] (0,0) circle (1);
	\draw[-,xshift=2cm] (115:1) arc [radius=1,start angle=115,end angle=155]; 
	\draw[-,xshift=2cm] (250:1) arc [radius=1,start angle=250,end angle=290]; 
	\draw[-,xshift=2cm] (25:1) arc [radius=1,start angle=25,end angle=65]; 
	
	\draw[xshift=2cm,->] (145:1) to[out=-35,in=100] (280:1); 
	\draw[xshift=2cm,->] (260:1) to[out=80,in=215] (35:1); 
	\draw[xshift=2cm,->] (125:1) to[out=-45,in=225] (55:1); 
	\end{scope}
\end{tikzpicture}}

\newcommand{\IntersectOp}{
\begin{tikzpicture}[scale=1,thick]
	\draw[decoration={markings,mark=at position 0.5 with {\arrow[scale=1.5]{stealth}}},
        postaction={decorate},opacity=.4,xshift=-5cm] (0,0) circle (1);
	\draw[xshift=-5cm] (25:1) arc [radius=1,start angle=25,end angle=-25];

	\draw[decoration={markings,mark=at position 0.0 with {\arrow[scale=1.5]{stealth}}},
        postaction={decorate},opacity=.4,xshift=-2cm] (0,0) circle (1);
	\draw[xshift=-2cm] (110:1) arc [radius=1,start angle=110,end angle=160];
	\draw[xshift=-2cm] (200:1) arc [radius=1,start angle=200,end angle=250];

	\draw[xshift=-2cm,stealth-] (125:1) to[out=-55,in=55] node[auto] {$g$} (235:1);

	\draw[-stealth] ([xshift=-5cm]10:1) to[out=10,in=145] node[auto] {$x_1$} ([xshift=-2cm]145:1);
	\draw[stealth-] ([xshift=-5cm]-10:1) to[out=-10,in=215] node[auto,swap] {$x_2$} ([xshift=-2cm]215:1);

	\draw[stealth-stealth,thick] (-.5,0) -- node[auto] {3b} (.5,0);

	\draw[decoration={markings,mark=at position 0.5 with {\arrow[scale=1.5]{stealth}}},
        postaction={decorate},opacity=.4,xshift=2cm] (0,0) circle (1);
	\draw[xshift=2cm] (25:1) arc [radius=1,start angle=25,end angle=-25];
	
	\draw[decoration={markings,mark=at position 0.0 with {\arrow[scale=1.5]{stealth}}},
        postaction={decorate},opacity=.4,xshift=5cm] (0,0) circle (1);
	\draw[xshift=5cm] (110:1) arc [radius=1,start angle=110,end angle=160];
	\draw[xshift=5cm] (200:1) arc [radius=1,start angle=200,end angle=250];

	\draw[xshift=5cm,stealth-] (145:1) to[out=-35,in=35] node[auto] {$g$} (215:1);
	
	\draw[stealth-] ([xshift=2cm]10:1) to[out=10,in=235] node[auto,swap] {$x_2$} ([xshift=5cm]235:1);
	\draw[-stealth] ([xshift=2cm]-10:1) to[out=-10,in=125] node[auto] {$x_1$} ([xshift=5cm]125:1);

\begin{scope}[yshift=-3.5cm]
	\draw[decoration={markings,mark=at position 0.5 with {\arrow[scale=1.5]{stealth}}},
        postaction={decorate},opacity=.4,xshift=-5cm] (0,0) circle (1);
	\draw[xshift=-5cm] (25:1) arc [radius=1,start angle=25,end angle=-25];

	\draw[decoration={markings,mark=at position 0.0 with {\arrow[scale=1.5]{stealth}}},
        postaction={decorate},opacity=.4,xshift=-2cm] (0,0) circle (1);
	\draw[xshift=-2cm] (110:1) arc [radius=1,start angle=110,end angle=160];
	\draw[xshift=-2cm] (200:1) arc [radius=1,start angle=200,end angle=250];

	\draw[xshift=-2cm,-stealth] (125:1) to[out=-55,in=35] node[auto] {$g$} (215:1);

	\draw[-stealth] ([xshift=-5cm]10:1) to[out=10,in=145] node[auto] {$x_1$} ([xshift=-2cm]145:1);
	\draw[-stealth] ([xshift=-5cm]-10:1) to[out=-10,in=235] node[auto,swap] {$x_2$} ([xshift=-2cm]235:1);

	\draw[stealth-stealth,thick] (-.5,0) -- node[auto] {3b} (.5,0);

	\draw[decoration={markings,mark=at position 0.5 with {\arrow[scale=1.5]{stealth}}},
        postaction={decorate},opacity=.4,xshift=2cm] (0,0) circle (1);
	\draw[xshift=2cm] (25:1) arc [radius=1,start angle=25,end angle=-25];
	
	\draw[decoration={markings,mark=at position 0.0 with {\arrow[scale=1.5]{stealth}}},
        postaction={decorate},opacity=.4,xshift=5cm] (0,0) circle (1);
	\draw[xshift=5cm] (110:1) arc [radius=1,start angle=110,end angle=160];
	\draw[xshift=5cm] (200:1) arc [radius=1,start angle=200,end angle=250];

	\draw[xshift=5cm,-stealth] (145:1) to[out=-35,in=55] node[auto] {$g$} (235:1);
	
	\draw[-stealth] ([xshift=2cm]10:1) to[out=10,in=215] node[auto,swap,near end] {$x_2$} ([xshift=5cm]215:1);
	\draw[-stealth] ([xshift=2cm]-10:1) to[out=-10,in=125] node[auto] {$x_1$} ([xshift=5cm]125:1);
\end{scope}
\end{tikzpicture}}

\begin{document}

\begin{abstract} 
A {\em virtual $n$-string} is a chord diagram with $n$ core circles and a collection of arrows between core circles. We consider virtual $n$-strings up to {\em virtual homotopy}, compositions of flat virtual Reidemeister moves on chord diagrams.

Given a virtual 1-string $\A$, Turaev~\cite{Turaev} associated a based matrix that encodes invariants of the virtual homotopy class of $\A$. We generalize Turaev's method to associate a multistring based matrix to virtual $n$-strings, addressing an open problem of Turaev and constructing similar invariants for virtual homotopy classes of virtual $n$-strings.\end{abstract}
\maketitle

\leftline {\em \Small 2010 Mathematics Subject Classification. Primary: 57M99}

\leftline{\em \Small Keywords: virtual homotopy, virtual strings, based matrices, flat virtual links}



\section{Introduction}

We work in the smooth ($C^\infty$) category and assume that all surfaces are oriented. Virtual multistrings are combinatorial representations of finite families of generic oriented closed curves on surfaces. The virtual homotopy classes of virtual multistrings corresponds to the theory of oriented flat virtual links, and so generalizing Turaev's based matrix to virtual multistrings induces new invariants for flat virtual links. These results address an open problem of Turaev from~\cite{Turaev}.

A {\em virtual $n$-string} $\B$ is a collection of $n$ oriented circles with $2m$ distinct points of the circles partitioned into $m$ ordered pairs. Each circle is a {\em core circle} of $\B$, each ordered pair of distinguished points is an {\em arrow} of $\B$, the collection of $2m$ distinguished points consists of the {\em endpoints} of $\B$, and the set of arrows is denoted by $\arr(\B)$. 

For an arrow $g=(a,b)$ of $\B$, the {\em tail} of $g$ is $g_t = a$ and the {\em head} of $g$ is $g_h = b$. So we may write $g = (g_t,g_h)$. Arrows are pictorially represented by an arrow directed from its tail to its head. An arrow is a {\em self-arrow} if both endpoints lie in the same core circle; an {\em intersection arrow} if they lie in different core circles. An example $2$-string is depicted in Figure~\ref{fig:2string}.

\begin{figure}
\begin{tikzpicture}[thick]
\begin{scope}[xshift=-1.5cm]
\draw[decoration={markings,mark=at position 0.755 with {\arrow[scale=1.5]{stealth}}},
        postaction={decorate}] (0,0) circle[radius=1];
\draw[stealth-] (200:1) to[out=20,in=110] (-70:1);
\draw[-stealth] (230:1) to[out=40,in=-90] (90:1);		
\end{scope}

\draw[-stealth] (-.5,0) -- (.5,0);

\begin{scope}[xshift=1.5cm]
\draw[decoration={markings,mark=at position 0.755 with {\arrow[scale=1.5]{stealth}}},
        postaction={decorate}] (0,0) circle[radius=1];
\draw[-stealth] (45:1) to[out=225,in=135] (-45:1);
\end{scope}
\end{tikzpicture}

\caption{Virtual 2-string with 3 self-arrows and 1 intersection arrow.}
\label{fig:2string}
\end{figure}
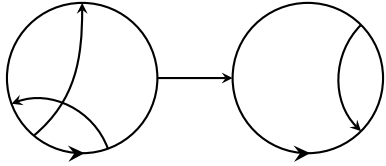

Given a virtual $n$-string $\B$, we may fix an order on the core circles. Then the $i$th core circle defines a unique virtual 1-string $\A_i$ obtained by defining $\arr(\A_i)$ to be all self-arrows of $\B$ lying in the $i$th core circle. Furthermore, we let $\arr_{i,j}(\B)$ denote the intersection arrows between $\A_i$ and $\A_j$ in $\B$, and $\arr_\cap(\B) = \ds{\bigcup_{i\neq j} \arr_{i,j}(\B)}$ the set of all intersection arrows of $\B$.

\subsection{Geometric interpretation of virtual $n$-strings}\label{sec:geom}

Given $n$ generic closed curves on a surface $\Sigma$, i.e., a generic immersion $\omega$ of $n$ oriented circles into $\Sigma$, we can construct a virtual $n$-string. Each component circle is a core circle and the arrows correspond to the double points of the immersion. More specifically, $(a,b)$ is an arrow if $\omega(a)=\omega(b)$ and the positive tangent vectors $v_a,v_b$ of $\omega$ at $a,b$ determine a positive basis of $T_{\omega(a)} \Sigma$.

Conversely, given a virtual $n$-string $\B$, we can realize $\B$ as a collection of $n$ generic closed curves on a surface $\Sigma_\B$ as in \cite{Turaev}. By identifying the endpoints of each arrow, we obtain a 4-regular graph $\Gam_\B$. Thickening $\Gam_\B$, we obtain a disk-band surface $\Sigma_\B$ in which $\B$ sits as the core of the surface. We refer to $\Sigma_\B$ as the {\it canonical surface} of $\B$.

\subsection{Flat virtual knot theory}

Virtual knot theory was introduced by Kauffman~\cite{Kauffman} and can be viewed as a generalization of the classical study of embeddings of circles into $\R^3$ (or the 3-sphere). In classical knot theory, we obtain diagrams by taking regular projections of links to either $\R^2$ or the 2-sphere and adding crossing data to distinguish over and under stands at double points. Kauffman built virtual knot theory by formally generalizing such diagrams, allowing for a new type of double point called a virtual crossing. Up to equivalence under a new collection of Reidemeister moves, these diagrams define virtual knots and links.

Virtual knot theory has two significant interpretations: as links in thickened orientable surfaces (considered up to stabilization and destabilization) and as the ``completion'' of Gauss codes. For a more completion description of each of these motivations, see~\cite{Carter,Kauffman}.

Flat virtual knot theory is constructed by ignoring the crossing data for classical crossings in virtual links (and the Reidemeister moves), thereby obtaining a {\em flat virtual link}. While flat knot theory (i.e., classical knot diagrams with no crossing data) is a trivial knot theory, flat virtual knot theory is nontrivial~\cite{Kauffman,Kadokami,Manturov}. Consequently, any invariant of flat virtual links is naturally promoted to an invariant of virtual links. Such links have been ascribed many names including ``projected virtual links'' (e.g.,~\cite{Kadokami}), ``shadows of virtual links'' (e.g.,~\cite{Kauffman}), and ``universes of virtual links'' (e.g.,~\cite{Carter}).

Carter, Kamada, and Saito~\cite{Carter} proved that flat virtual links correspond to stable equivalence classes of loops on surfaces. That is, a flat virtual link diagram can be realized as a family of loops on a surface and, under the flat virtual Reidemeister moves, this family deforms via homotopy and the surface undergoes stabilization and destabilization (adding or removing handles disjoint from the link). While their result ostensibly applies to virtual links in thickened surfaces, it immediately descends to a proof for flat virtual links on surfaces~\cite{Kadokami}.

Virtual multistrings are best described as oriented flat virtual links. This relationship follows immediately from the Gauss code description of a virtual link. Every oriented flat virtual link may be realized by a signed Gauss code without specifying the ``over'' and ``under'' structures. In this way, we obtain a Gauss code where each crossing is specified by whether the positive tangent vectors to the curve are consistent (contributing the ``$+$'' strand of the crossing) or inconsistent (the ``$-$'' strand) with the orientation of the plane. This is precisely an arrow of a virtual string with a ``$+$'' being the tail of an arrow; ``$-$'' the head of an arrow.

\subsection{Virtual homotopy}

In this section, we follow terminology used by Turaev~\cite{Turaev} and Cahn~\cite{Cahn}. A virtual homotopy is an analogue of Reidemeister moves for virtual multistrings. Two virtual $n$-strings are {\em virtually homotopic} if they are related by a finite sequence of the following moves and their inverses:
\begin{itemize}[leftmargin=*]
\item Type 1: Given an arc $ab$ of a core circle containing no endpoints, add an arrow $(a,b)$ (Type 1a) or an arrow $(b,a)$ (Type 1b).
\item Type 2: Let the pairs $\set{a,a'}$ and $\set{b,b'}$ define two disjoint arcs of core circles, each containing no endpoints. Add arrows $(a,b)$ and $(b',a')$.
\item Type 3a: Let $aa^+$, $bb^+$, $cc^+$ be three disjoint arcs of core circles, each containing no endpoints. If $(a^+,b),(b^+,c),(c^+,a)$ are arrows, then replace them with $(a,b^+),(b,c^+),(c,a^+)$.
\item Type 3b: Let $aa^+$, $bb^+$, $cc^+$ be three disjoint arcs of core circles, each containing no endpoints. If $(a,b),(a^+,c),(b^+,c^+)$ are arrows, then replace them with $(a^+,b^+),(a,c^+),(b^+,c^+)$.
\end{itemize}

While the Type 1 moves only apply to an individual core circle, the remainder of the virtual homotopies describe variants for both self-arrows and intersection arrows. For instance, there are 8 versions of the Type 2 move depending on both the order of points and whether the arcs are part of the same core circle. Furthermore, the Type 3b move is redundant as it can be obtain from a sequence of Type 3a and Type 2 moves~\cite[Section 2.3]{Turaev}. Example virtual homotopies are depicted in Figure~\ref{fig:homotopies}. Let $[\B]$ denote the virtual homotopy class of a virtual $n$-string $\B$.

\begin{figure}
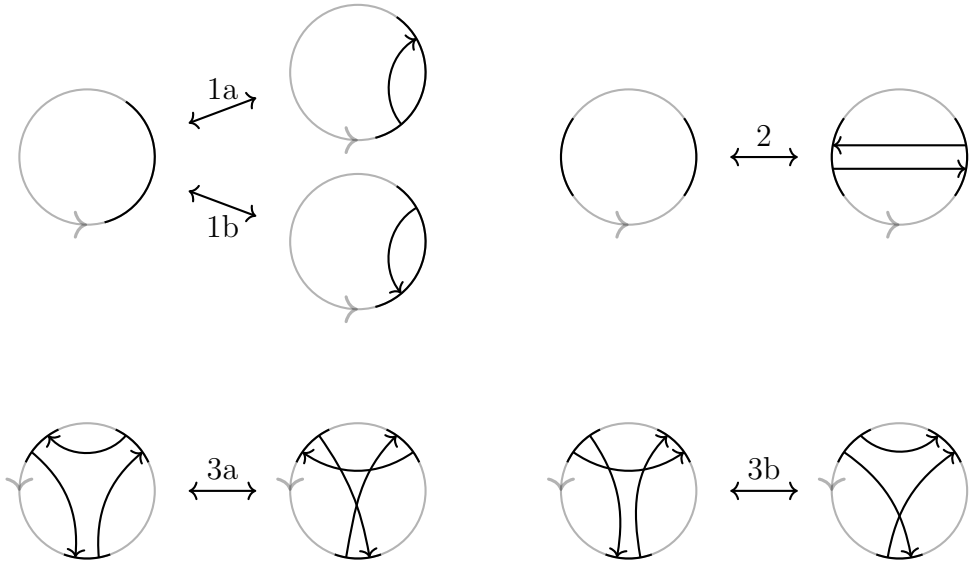

\CRMoves{scale=.9,thick}
\caption{Virtual homotopies.}
\label{fig:homotopies}
\end{figure}

\subsection{Based matrices}

Turaev~\cite{Turaev} defined based matrices as a means to algebraically study virtual strings. Using the canonical surface of a virtual string (see Section~\ref{sec:geom}), based matrices arise as the Gram matrix of the homological intersection pairing on the surface with respect to a canonical basis. By considering the transformations on the Gram matrix under virtual homotopies of the underlying virtual string, Turaev defined corresponding operations on the matrix.

Formally, Turaev defined a {\em based matrix} over an abelian group $H$ to be a triple $(G,s,b)$ where $G$ is a finite set, $s\in G$, and $b:G\times G\to H$ is skew-symmetric (i.e., $b(g,h) = -b(h,g)$ and $b(g,g)=0$ for all $g,h\in G$). Two based matrices $(G,s,b)$ and $(G',s',b')$ are {\em isomorphic} if there is a bijection $\p:G\to G'$ such that $\p(s)=s'$ and $b'(\p(g),\p(h)) = b(g,h)$ for all $g,h\in G$.

In order to define an equivalence relation on based matrices, we distinguish three types of elements of $G$:
\begin{itemize}[leftmargin=*]
\item Annihilating elements: An element $g\in G\setminus\set{s}$ is {\em annihilating} if $b(g,h)=0$ for all $h\in G$.
\item Core elements: An element $g\in G\setminus\set{s}$ is {\em core} if $b(g,h)=b(s,h)$ for all $h\in G$.
\item Complementary elements: Two elements $g_1,g_2\in G\setminus\set{s}$ are {\em complementary} if $b(g_1,h)+b(g_2,h) = b(s,h)$ for all $h\in G$.
\end{itemize}

\subsubsection{Based matrix of a virtual string}

Specializing to the case of a virtual string $\A$, Turaev~\cite{Turaev} defined a based matrix $T(\A) = (G,s,B)$ by considering the canonical surface $\Sigma_\A$ associated to $\A$. Let $G = \set{s}\cup\arr(\A)$, where $s$ is the homotopy class of a closed curve in $\Sigma_\A$ representing $\A$. More specifically, each arrow $(a,b)\in\arr(\A)$ defines a unique homotopy class of loops in $\Sigma_\A$ and the collection $G = \set{s}\cup\arr(\A)$ corresponds to a basis of $H_1(\Sigma_\A)$ (see Section 4.2 of~\cite{Turaev}). Since $\Sigma_\A$ is an oriented surface, the orientation induces a homological intersection pairing $b:H_1(\Sigma_\A)\times H_1(\Sigma_\A)\to \Z$. Restricting to the basis identified with $G$, we obtain a skew-symmetric map $B = b|_{G\times G}$. 

Turaev~\cite{Turaev} proved that the combinatorial structure of virtual strings induces a combinatorial formula for $B(g,h)$ for all $g,h\in G$. Given arrows $g=(a,b)$ and $h=(c,d)$ of $\A$, we say $h$ {\em links} $g$ if exactly one endpoint of $h$ is in the open arc $\ov{g}$ from $a$ to $b$. Consistent with~\cite{Cahn,Turaev}, we say this is a {\em positive} linking if $c\in \ov{g}$; a {\em negative} linking if $d\in \ov{g}$. Visually, as shown in Figure~\ref{fig:linking}, $h$ links $g$ positively if rotating the arrow $g$ according to the orientation of the core lines up the arrows; negatively if they point in opposite directions after rotating. If $h$ does not link $g$, then we say $h$ and $g$ are {\em unlinked}.

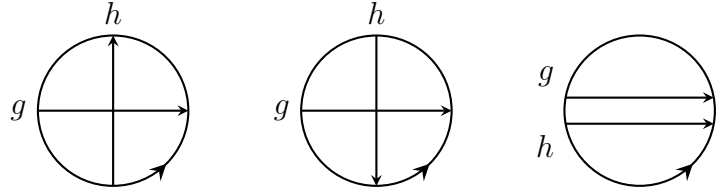
\begin{figure}
	\begin{tikzpicture}[thick]
	\draw[decoration={markings,mark=at position 0.875 with {\arrow[scale=1.5]{stealth}}},
        postaction={decorate}] (0,0) circle[radius=1cm];
	\draw[-stealth] (180:1cm) node[left]{$g$} -- (0:1cm);
	\draw[-stealth] (-90:1cm) -- (90:1cm) node[above]{$h$};
	
	\begin{scope}[xshift=3.5cm]
	\draw[decoration={markings,mark=at position 0.875 with {\arrow[scale=1.5]{stealth}}},
        postaction={decorate}] (0,0) circle[radius=1cm];
	\draw[-stealth] (180:1cm) node[left]{$g$} -- (0:1cm);
	\draw[stealth-] (-90:1cm) -- (90:1cm) node[above]{$h$};
	\end{scope}
	
	\begin{scope}[xshift=7cm]
	\draw[decoration={markings,mark=at position 0.875 with {\arrow[scale=1.5]{stealth}}},
        postaction={decorate}] (0,0) circle[radius=1cm];
	\draw[-stealth] (170:1cm) node[above left]{$g$} -- (10:1cm);
	\draw[-stealth] (190:1cm) node[below left]{$h$} -- (-10:1cm);
	\end{scope}
	\end{tikzpicture}
\caption{From left-to-right: $h$ links $g$ positively, $h$ links $g$ negatively, $h$ and $g$ unlinked.}
\label{fig:linking}
\end{figure}

Let $n(g)$ denote the algebraic intersection number of the loop defined by $g$, i.e., the sum of signed intersection points between $g$ and the curve $\A$. Then Turaev~\cite{Turaev} proved \[n(g)= \#\set{h\in\arr(\A): h \text{ links $g$ positively}} - \#\set{h\in\arr(\A): h \text{ links $g$ negatively}}\] and showed that $\boxed{B(g,s) = n(g)}$. For $g,g'\in G\setminus\set{s}$, define \[g\cdot g' = \#\set{ h\in \arr(\A): h_t\in \ov{g} \text{ and } h_h\in \ov{g'}} - \#\set{h\in \arr(\A): h_h\in \ov{g} \text{ and } h_t\in \ov{g'}}.\] By Lemma 4.2.1 in~\cite{Turaev},
\[\boxed{B(g,g') = g\cdot g' + \E}, \text{ where } \E = \begin{cases} 0 & g,g' \text{ unlinked}\\ 1 & g' \text{ links $g$ positively}\\ -1 & g' \text{ links $g$ negatively}\end{cases}.\]

For arrows $g=(a,b), g'=(c,d)$, Turaev~\cite{Turaev} and Cahn~\cite{Cahn} write $g\cdot g' = ab\cdot cd$. Our notation has advantages in Section 2.

\subsubsection{Homologous based matrices}

Turaev~\cite{Turaev} defined operations $M_1$, $M_2$, and $M_3$, called {\em elementary extensions}, on based matrices:

\begin{itemize}[leftmargin=*]
\item $M_1$ adds an annihilating element $g$: For a based matrix $(G,s,b)$, form the based matrix $(G',s,b')$ by taking $G' = G\cup\set{g}$, $b'|_{G\times G} = b$, and $b'(g,h)=0$ for all $h\in G'$.

\item $M_2$ adds a core element $g$: For a based matrix $(G,s,b)$, form the based matrix $(G',s,b')$ by taking $G'= G\cup\set{g}$, $b'|_{G\times G} = b$, and $b'(g,h)=b(s,h)$ for all $h\in G'$.

\item $M_3$ adds a pair of complementary elements $g_1,g_2$: For a based matrix $(G,s,b)$, form a based matrix $(G',s,b')$ by taking $G' = G\cup\set{g_1,g_2}$, $b'|_{G\times G} = b$, and $b'(g_1,h) + b'(g_2,h) = b(s,h)$ for all $h\in G'$.
\end{itemize}

Two based matrices are {\em homologous} if they are related by a finite sequence of isomorphisms, and elementary extensions and their inverses. Turaev~\cite{Turaev} proved that virtually homotopic strings have homologous based matrices.

Note that there is a degree of freedom when applying $M_3$ moves. That is, the condition for a complementary pair does not completely determine the individual values of the pairing. This freedom makes equivalence on based matrices a weaker relation than virtual homotopy (i.e., the homology class of $T(\A)$ is larger than the virtual homotopy class of $\A$).

In each homology class of based matrices there are representatives which act as minimal representatives. A {\em primitive based matrix} is a based matrix $T_\bullet = (G_\bullet,s_\bullet,b_\bullet)$ such that no element of $G_\bullet$ is annihilating, core, or part of a complementary pair. That is, no $M_i^{-1}$ moves can be applied to $T_\bullet$.

Given a based matrix, a homologous primitive based matrix can be obtained by repeated application of $M_i^{-1}$ moves. Turaev~\cite[Lemma 6.1.1]{Turaev} proved that any two homologous primitive based matrices are related by an isomorphism of based matrices.

\subsubsection{Invariants for virtual strings}
\label{sec:invariants}

Based matrices can be used to construct invariants for virtual homotopy classes of virtual strings. Let $\A$ be a virtual string and $T(\A) = (G,s,B)$ the associated based matrix.

For $m\in\Z$, define $\sign(m) = \begin{cases} 1 & m>0\\ 0 & m=0\\ -1 & m<0\end{cases}$. The {\em $u$-polynomial} of $[\A]$ is $\ds{u([\A]) = \sum_{g\in \arr(\A)} \sign(n(g))t^{\abs{n(g)}}}$. Turaev~\cite{Turaev} proved that $u([\A])$ does not depend on the representative $\A$, hence defining an invariant of $[\A]$. Since $n(g) = B(g,s)$, the $u$-polynomial may be expressed entirely in terms of the based matrix.

Among the various properties of $u([\A])$, we have the following:
\begin{itemize}[leftmargin=*]
\item For any $\B$ in $[\A]$, $\deg u([\A]) + 1 \leq \#\arr(\B)$.
\item If $\ov{\A}$ is the virtual string obtained by reversing all arrows, then $u([\ov{\A}]) = u([\A])$.
\item If $\A^{-}$ is the virtual string obtained by reversing the orientation of the core circle, then $u([\A^-]) = -u([\A])$.
\item A polynomial $u\in \Z[t]$ can be realized as the $u$-polynomial of a virtual string if and only if $u(0) = u'(1) = 0$.
\end{itemize}

Additional invariants arise from the primitive based matrix $T_\bullet(\A) = (G_\bullet,s_\bullet,B_\bullet)$ associated to $T(\A)$. Turaev~\cite{Turaev} defined the $\rho$ invariant to be $\rho([\A]) = \# G_\bullet -1$. Then $\rho([\A])\leq \#\arr(\B)$ for any $\B$ in $[\A]$. Moreover, for an infinite family of examples, Turaev~\cite[Section 6.3]{Turaev} proved that $\rho$ can detect the non-triviality of virtual strings for which the $u$-polynomial vanishes.

Let $g(\A)$ be the minimal genus of a surface supporting $\A$. Then $g([\A]) = \ds{\min_{\B\in [\A]} g(\B)}$ is an invariant of $[\A]$. Turaev~\cite{Turaev} proved that $g([\A]) \geq \frac{1}{2}\operatorname{rank} B_\bullet$. This follows from the choice of $B$ as the homological intersection pairing.

In particular, if $T(\A)$ is isomorphic to $T_\bullet(\A)$, then the bounds on genus and the number of self-intersections are sharp. Such virtual strings are thus already minimal in their virtual homotopy class with respect to the number of self-intersections and genus.

\subsection{Organization}

In Section 2, we define woven based matrices and an equivalence relation on them, associate a canonical woven based matrix to a virtual $n$-string (a multistring based matrix) and show that the operations on woven based matrices arise from virtual homotopies, and determine algebraic consequences of the homology operations on woven based matrices. Section 3 is devoted to proving that there is a unique minimal representative in the homology class of a woven based matrix. Finally, in Section 4, we produce invariants of virtual multistrings arising from their based matrices, including a generalizations the $u$-polynomial and the $\rho$ invariant.

\section{Woven Based Matrices}

In this section, we generalize based matrices to virtual multistrings by using woven based matrices. Let $\set{(G_i,s_i,b_i)}_{i=1}^n$ be a collection of based matrices over an abelian group $H$ with $G_i$ disjoint finite sets, $G = G_1\cup\dots\cup G_n$, and $I$ a finite set disjoint from $G$. We refer to $I$ as the {\it weaving set}. A function $\p: (G\cup I)\times I\to H$ is a {\em weaving map} if:
\begin{itemize}[leftmargin=*]
\item For $x\in I$, there are exactly two indices $i,j$ such that $\p(s_i,x),\p(s_j,x)\neq 0$. For this unique pair of indices, we require $\p(s_i,x) = -\p(s_j,x)$.
\item $\p(g,x) \in \set{0, \p(s_i,x)}$ for all $g\in G_i$ and $x\in I$.
\item $\p(x,y) = 0$ for all $x,y\in I$.
\end{itemize}
Weaving maps will be our main tool to take multiple based matrices and weave them together using the elements of $I$. 
Given a weaving map $\p$, a {\em based matrix woven by $\p$} is a $(2n+2)$-tuple $(G_i,I,s_i,b)$ where $b: (G\cup I)\times (G\cup I)\to H$ is a skew-symmetric map such that:
\begin{itemize}[leftmargin=*]
\item $b|_{G_i\times G_i} = b_i$ for each $i$,
\item $b|_{(G\cup I)\times I}=\p$.
\end{itemize}

Let $(G_i,I,s_i,b)$ and $(G_i',I',s_i',b')$ be two woven based matrices. A bijection $\Phi: G\cup I \to G'\cup I'$ is an {\em isomorphism} of woven based matrices if there is a permutation $\sigma$ of $\set{1,\dots,n}$ such that:
\begin{itemize}[leftmargin=*]
\item $\Phi(s_i) = s_{\sigma(i)}'$ for all $i=1,\dots,n$,
\item $\Phi(G_i) = G_{\sigma(i)}'$ for all $i=1,\dots,n$, and
\item $b'(\Phi(a_1),\Phi(a_2)) = b(a_1,a_2)$ for all $a_1,a_2\in G\cup I$.
\end{itemize}
That is, $\Phi^\ast b' = b$ and, for all $i$, $\Phi$ restricts to isomorphisms of based matrices $\Phi_i: (G_i,s_i,b_i)\to (G_{\sigma(i)}',s_{\sigma(i)}',b_{\sigma(i)}')$.

\subsection{Based matrix of a virtual multistring}

We now associate a particular woven based matrix to a virtual $n$-string $\B$. In the case $n\geq 2$, Turaev's method of using the canonical surface $\Sigma_\B$ to construct a based matrix fails. That is, there is no canonical means of producing $\Z$-independent cycles associated to the intersection arrows.
Instead, we define a canonical weaving map for $\B$ and thus associate a woven based matrix.

Fix an ordering of the core circles of $\B$ and let $\A_1,\dots,\A_n$ be the induced virtual strings. Then $T(\A_i)=(G_i,s_i,B_i)$ is the based matrix associated to $\A_i$ and we define $G= G_1\cup\dots\cup G_n$, $I= \arr_\cap (\B)$. Consider the canonical surface $\Sigma_\B$ for $\B$ and the homological intersection pairing $b$ on it. For each self-arrow $g$ of $\B$, define $[g]$ as the the cycle defined by projecting the oriented arc $g_tg_h$ to $\Sigma_\B$.

For elements $g,g'\in G_i$, $b(g,g') = B_i(g,g')$. For elements $g_i\in G_i$ and $g_j\in G_j$ with $i\neq j$, we show that $b(g_i,g_j) = g_i\cdot g_j$. Indeed, the cycles $[g_i]$ and $[g_j]$ intersect transversely in $\Sigma_\B$, and so $b(g_i,g_j)$ counts the intersection points with signs. This is precisely $g_i\cdot g_j$.

Finally, define a map $\p_\B: (G\cup I)\times I\to \Z$ by
\[\p_\B(y,x)=0, \p_\B(s_i,x) = \begin{cases} 1 & x_t\in \A_i\\ -1 & x_h\in \A_i\\ 0 & x_t,x_h\notin \A_i\end{cases}, \text{ and } \p_\B(g,x) = \begin{cases} 1 & x_t\in \ov{g}\\ -1 & x_h\in \ov{g}\\ 0 & x_t,x_h\notin \ov{g}\end{cases}\]
for $g\in G$, $x,y\in I$. Since intersection arrows join exactly two core circles, $\p_\B$ is a weaving map. Thus the {\em multistring based matrix} $T(\B) = (G_i,I,s_i,B)$ associated to $\B$ is the based matrix woven by $\p_\B$ such that $B|_{G\times G} = b$.

Intuitively, the skew-symmetric map $B$ pairs a self-arrow $g$ with an intersection arrow $x$ by considering whether the tail or head of $x$ lies in the arc $\ov{g}$. Restricted to any collection of self-arrows for a single core circle, we recover the homological intersection form for the associated virtual string.

\subsection{Examples}

Turaev~\cite{Turaev} defined a family of 1-strings $\A_{p,q}$ for $p,q\in \Z_{\geq 0}$ by taking $p$ parallel arrows pointing from left-to-right and $q$ parallel arrows from bottom-to-top. Turaev~\cite{Turaev} proved that the based matrix $T(\A_{p,q})$ is primitive when $p,q>0$ and $p+q\geq 3$.

We generalize this construction to a family of 2-strings $\B=\B(p_1,p_2,q_1,q_2,r,s)$ by letting $\A_1 = \A_{p_1,q_1}, \A_2 = \A_{p_2,q_2}$, and letting $\B= \A_1\sqcup \A_2$ with $r+s$ intersection arrows as depicted in Figure~\ref{fig:betafamily}. Consider elements $g_{p,i},g_{q,j},h_{p,k},h_{q,\ell}, x_m,y_n$ where $g_{p,i},h_{p,k}$ are $p_1$ and $p_2$ arrows, $g_{q,j},h_{q,\ell}$ are $q_1$ and $q_2$ arrows, $x_m,y_n$ are $r$ and $s$ arrows. Then
\[T(\B) = \begin{tabular}{c|cccccccc}
& $s_1$ & $g_{p,i}$ & $g_{q,j}$ & $s_2$ & $h_{p,k}$ & $h_{q,\ell}$ & $x_m$ & $y_n$\\\hline
$s_1$ & & & & $r-s$ & $-s$ & $r$ & $1$ & $-1$ \\
$g_{p,i}$ & & $T(\A_1)$ & & $-s$ & $-s$ & $0$ & $0$ & $-1$\\ 
$g_{q,j}$ & & & & $r$ & $0$ & $r$ & $1$ & $0$\\
$s_2$ & $s-r$ & $s$ & $-r$ & & & & $-1$ & $1$\\
$h_{p,k}$ & $s$ & $s$ & $0$ & & $T(\A_2)$ & & $0$ & $1$\\
$h_{q,\ell}$ & $-r$ & $0$ & $-r$ & & & & $-1$ & $0$\\
$x_m$ & $-1$ & $0$ & $-1$ & $1$ & $0$ & $1$ & $0$ & $0$\\
$y_n$ & $1$ & $1$ & $0$ & $-1$ & $-1$ & $0$ & $0$ & $0$\\
\end{tabular}.\]

	\begin{figure}
	\begin{tikzpicture}[scale=1,thick,-stealth]
	\draw[decoration={markings,mark=at position 0.875 with {\arrow[scale=1.5]{stealth}}},
        postaction={decorate},xshift=-1.5cm] (0,0) circle[radius=1cm];
	\draw (-2.5cm,0) -- (-.5cm,0) node[below left]{$p_1$};
	\draw (-1.5cm,-1cm) -- (-1.5cm,1cm) node[below left]{$q_1$};
				
	\draw ([xshift=-1.5cm] 45:1cm) to[out=45,in=45] node[auto]{$r$} ([xshift=1.5cm] 45:1cm);
	\draw ([xshift=1.5cm] 225:1cm) to[out=225,in=225] node[auto]{$s$} ([xshift=-1.5cm] 225:1cm);
	
	\draw (.5cm,0) -- (2.5cm,0) node[below left]{$p_2$};
	\draw (1.5cm,-1cm) -- (1.5cm,1cm) node[below left]{$q_2$};
				
	\draw[decoration={markings,mark=at position 0.875 with {\arrow[scale=1.5]{stealth}}},
        postaction={decorate},xshift=1.5cm] (0,0) circle[radius=1cm];
	\end{tikzpicture}
	\caption{The family of $2$-strings $\B(p_1,q_1,p_2,q_2,r,s)$.}
	\label{fig:betafamily}
	\end{figure}
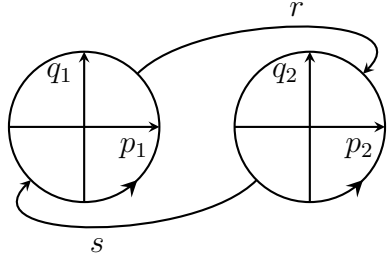

\subsection{Homologous woven based matrices}

We want an equivalence on woven based matrices that captures the effects of virtual homotopies on virtual multistrings. Let $(G_i,I,s_i,b)$ be a based matrix woven by $\p$ and $G= G_1\cup\dots\cup G_n$. Then we distinguish six types of elements:

\begin{itemize}[leftmargin=*]
\item Annihilating elements: An element $g\in G_i\setminus\set{s_i}$ is {\em $G_i$-annihilating} if $b(g,a)=0$ for all $a\in G\cup I$.
\item Core elements: An element $g\in G_i\setminus\set{s_i}$ is {\em $G_i$-core} if $b(g,a) = b(s_i,a)$ for all $a\in G\cup I$.
\item Complementary elements: Elements $g_1,g_2\in G_i\setminus\set{s_i}$ are {\em $G_i$-complementary} if $b(g_1,a) + b(g_2,a) = b(s_i,a)$ for all $a\in G\cup I$.
\item Sum-annihilating elements: A pair of elements $x_1,x_2 \in I$ are {\em sum-annihilating} if $b(x_1,a) + b(x_2,a) = 0$ for all $a\in G\cup I$.

\item A pair of elements $x_1,x_2\in I$ are {\em $g$-annihilating} for $g\in G_i\setminus\set{s_i}$ if
\[b(g,x_1) + b(g,x_2) = 0 = b(s_i,x_1) + b(s_i,x_2).\]

\item A pair of elements $x_1,x_2\in I$ are {\em $g$-unequal} for $g\in G_i\setminus\set{s_i}$ if 
\[b(g,x_1) \neq b(g,x_2) \text{ and } b(s_i,x_1) = b(s_i,x_2).\]
\end{itemize}

\begin{rem} The $G_i$-annihilating, $G_i$-core, and $G_i$-complementary elements are direct generalizations of Turaev's distinguished elements and each involve elements of a specific $G_i$. 
Meanwhile, sum-annihilating, $g$-annihilating, and $g$-unequal elements only exist in the weaving set $I$.
\end{rem}

\subsubsection{Moves on woven based matrices}

Generalizing Turaev's elementary extensions, we define an extended list of {\em elementary extensions} for woven based matrices:
\begin{itemize}[leftmargin=*]
\item $M_{1,j}(g)$ adds a $G_j$-annihilating element $g$: Form the woven based matrix \[(G_1,\dots, G_{j-1}, G_j', G_{j+1},\dots G_n, I, s_i, b')\] by taking $G_j' = G_j\cup\set{g}$ and letting $b'$ be the skew-symmetric map extending $b$ such that $b'(g,a) = 0$ for all $a$.

\item $M_{2,j}(g)$ adds a $G_j$-core element $g$: Form the woven based matrix \[(G_1,\dots, G_{j-1}, G_j', G_{j+1},\dots G_n, I, s_i, b')\] by taking $G_j' = G_j\cup\set{g}$ and letting $b'$ be the skew-symmetric map extending $b$ such that $b'(g,a) = b(s_j,a)$ for all $a$.

\item $M_{3,j}(g_1,g_2)$ add a $G_j$-complementary pair of elements $g_1,g_2$: Form the woven based matrix \[(G_1,\dots, G_{j-1}, G_j', G_{j+1},\dots G_n, I, s_i, b')\] by taking $G_j' = G_j\cup\set{g_1,g_2}$ and letting $b'$ be a skew-symmetric map extending $b$ such that:
	\begin{enumerate}[1.]
	\item $b(g_1,a) + b(g_2,a)=b(s_i,a)$ for all $a\in \underbrace{G_1\cup\dots\cup G_j'\cup\dots\cup G_n}_{G'}\cup I$,
	\item $b|_{(G'\cup I) \times I}$ is a weaving map.
	\end{enumerate}

\item $M_4(x_1,x_2)$ adds a sum-annihilating pair of elements $x_1,x_2$: Form the woven based matrix $(G_i,I',s,b')$ by taking $I' = I\cup\set{x_1,x_2}$ and letting $b'$ be a skew-symmetric map extending $b$ such that $b'|_{(G\cup I')\times I'}$ is a weaving map.
\end{itemize}

\begin{rem} Unless the distinction becomes necessary, we denote the first three operations by $M_1,M_2,M_3$, omitting the subscript indicating the set $G_j$.
\end{rem}


This list of elementary extensions is insufficient to capture all virtual homotopies for multistring based matrices. We additionally require the following {\em intersection moves} which, in the specific setting of interest, arise from Type 3 moves:

\begin{itemize}[leftmargin=*]
\item $I_1(g;x_1,x_2)$ alters values of a $g$-annihilating pair $x_1,x_2$ for $g\in G_j\setminus\set{s_j}$: Let $x_1,x_2\in I$ be $g$-annihilating. Then form the woven based matrix $(G_i,I,s_i,b')$ such that 
\[b'(g,x_1) = b(s_j,x_1)-b(g,x_1) \text{ and } b'(g,x_2) = b(s_j,x_2)-b(g,x_2)\]
and agrees with $b$ on all other pairs of elements.  

\item $I_2(g;x_1,x_2)$ alters values of a $g$-unequal pair $x_1,x_2$ for $g\in G_j\setminus\set{s_j}$: Let $x_1,x_2\in I$ be $g$-unequal. Then form the woven based matrix $(G_i,I,s_i,b')$ such that 
\[b'(g,x_1) = b(s_j,x_1)-b(g,x_1) \text{ and } b'(g,x_2) = b(s_j,x_2)-b(g,x_2)\]
and agrees with $b$ on all other pairs of elements. 
\end{itemize}

\begin{rem} In the interest of simplicity, we refer to elementary extensions and intersection moves without specifying the elements involved whenever possible.
\end{rem}

\begin{rem} Intersection moves act by ``switching'' the values on $g$-annihilating and $g$-unequal pairs. By definition of a weaving map, we know that $\p(g,x) \in \set{0,\p(s_j,x)}$, and applying $I_1$ or $I_2$ switches to the other element.\end{rem}

Working back through the definitions, 
intersection moves are representatively modeled by the following example:

\begin{ex}
Let $T=(G_i,I,s_i,b)$ be a woven based matrix over $\Z$ with $g\in G_1\setminus\set{s_1}$ and $x_1,x_2,x_3\in I$ such that, on the specified elements, the map $b$ is given by \begin{tabular}{c|ccc}
$b$ & $x_1$ & $x_2$ & $x_3$\\ \hline
$s_1$ & $1$ & $-1$ & $1$ \\
$g$ & $0$ & $0$ & $1$
\end{tabular}. Then $x_1,x_2$ form a $g$-annihilating pair and $x_1,x_3$ form a $g$-unequal pair.

Applying $I_1(g;x_1,x_2)$ to $T$, the resulting $T'$ has skew-symmetric map $b'$ given by \begin{tabular}{c|ccc}
$b'$ & $x_1$ & $x_2$ & $x_3$\\ \hline
$s_1$ & $1$ & $-1$ & $1$ \\
$g$ & $1$ & $-1$ & $1$
\end{tabular}. 

Instead applying $I_2(g;x_1,x_3)$ to $T$, we obtain $T''$ with skew-symmetric map $b''$ given by \begin{tabular}{c|ccc}
$b''$ & $x_1$ & $x_2$ & $x_3$\\ \hline
$s_1$ & $1$ & $-1$ & $1$ \\
$g$ & $1$ & $0$ & $0$
\end{tabular}.

After applying an intersection move, it is possible for new intersection moves to become available (e.g., $x_2,x_3$ are $g$-annihilating for both $T'$ and $T''$ but not $T$).
\end{ex}

\subsubsection{Homologous woven based matrices}

Analogous to the definition of Turaev's based matrices, two woven based matrices are {\em homologous} if they are related by a finite sequence of elementary extensions (and their inverses), intersection moves, and isomorphisms. We now prove that the definitions of elementary extensions and intersection moves are motivated by virtual homotopies of virtual multristrings. That is, we show the following:

\begin{prop}[c.f. Lemma 6.2.1 in \cite{Turaev}]\label{prop:homotopy} Let $\B,\B'$ be virtually homotopic virtual $n$-strings. Then the associated multistring based matrices $T(\B)$ and $T(\B')$ are homologous.
\end{prop}

\begin{proof} Fix an order on the core circles of $\B$, say $\A_1,\dots,\A_n$. It suffices to consider $\B'$ as the result of applying a single move to $\B$:
\begin{itemize}[leftmargin=*]
\item Type 1: Consider an arc $ab$ of $\A_j$ which contains no endpoints of $\B$. Adding an arrow $(a,b)$ corresponds to adding a $G_j$-annihilating element; adding $(b,a)$ corresponds to adding a $G_j$-core element. That is, a Type 1 move on $\B$ amounts to applying $M_{1,j}$ or $M_{2,j}$ to $T(\B)$.

\item Type 2: Let $\set{a,a'}$ and $\set{b,b'}$ define disjoint arcs, each containing no endpoints of $\B$. If both arcs lie in the same core circle $\A_j$, then the arrows $(a,b)$ and $(b',a')$ define a $G_j$-complementary pair. Hence this version of Type 2 move on $\B$ results in a $M_{3,j}$ move on $T(\B)$.

If the arcs lie in different core circles, then we claim that the arrows $x_1=(a,b),x_2=(b',a')$ form a sum-annihilating pair. For a self-arrow $g$, if $\ov{g}$ contains the endpoint of $x_1$ or $x_2$, then it must contain an endpoint of the other arrow as well. As the arrows are necessarily pointed in opposing directions, $B(s_i,x_1)=-B(s_i,x_2)$ for all $i$. Thus $x_1,x_2$ are sum-annihilating and the Type 2 move on $\B$ induces a $M_4$ move on $T(\B)$.

\item Type 3: Let $aa^+,bb^+,cc^+$ denote the disjoint arcs, none containing endpoints of $\B$. If all arcs belong to the same core circle, then the Type 3 move induces an isomorphism of woven based matrices.

If all arcs belong to distinct core circles, then the Type 3 move switches the order of the intersection arrows on each core circle. However, for the construction of multistring based matrices, the value $B(g,x)$ only depends on whether an endpoint of $x$ is in the arc $\ov{g}$ and not its position within the arc. As this version of Type 3 move doesn't affect self-arrows, the multistring based matrix is unchanged.

Finally, suppose that two arcs belong to the same core circle while the third lies in a different core circle. Then we encounter one of the situations depicted in Figure~\ref{fig:int_move} (or an analogous setup for Type 3a moves). In this case, the Type 3 move induces an intersection move $I_\square$ on $T(\B)$ with $\square=1,2$.

\begin{figure}
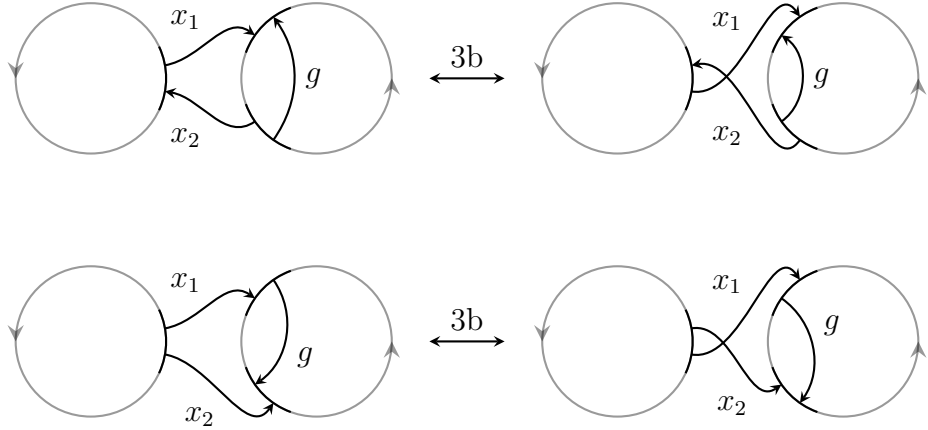

\IntersectOp
\caption{Type 3b moves inducing different intersection moves (top: $I_1(g)$; bottom: $I_2(g)$).}
\label{fig:int_move}
\end{figure}
\end{itemize}
In all cases, $T(\B)$ and $T(\B')$ are homologous.\end{proof}

\begin{rem} Notice that the $I_1,I_2$ moves can each arise from a Type 3a or Type 3b move, depending on the configuration of self-arrow with intersection arrows.\end{rem}

\subsubsection{Properties of intersection moves}

For later use, we establish algebraic properties of the intersection moves for woven based matrices. In particular, Proposition~\ref{prop:reduced} leads to the useful notion of a reduced sequence of intersection moves. Throughout, let $T = (G_i,I,s_i,b)$ be a woven based matrix.

\begin{rem}\label{rem:involution} For $\square=1,2$, we have $I_\square(g;b_1,b_2)\circ I_\square(g;b_1,b_2) = \Id$. That is, $I_\square$ is an involution when applied to the same triple of elements.\end{rem}

For sequences of intersection moves, it is possible to obtain the same woven based matrix by means of a different sequence of intersection moves. To make this more formal, let $N,N'$ be two sequences of intersection moves on a woven based matrix $T$ such that the resulting based matrices are equal. Then we may replace $N$ by $N'$, denoted $N\rightarrow N'$.

\begin{lem} \label{lem:props} The following hold:

\begin{enumerate}[(i)]
\item (Actions on distinct elements commute) If $g,g'\in G$ are distinct and $\square,\triangle\in\set{1,2}$, then \[I_\square(g)\circ I_\triangle(g') = I_\triangle(g')\circ I_\square(g).\]
\item (Actions on distinct weaving set pairs commute) If $g\in G$, $\square,\triangle\in\set{1,2}$, and $\set{x_1,x_2}\cap\set{x_3,x_4}=\es$, then \[I_\square(g;x_3,x_4)\circ I_\triangle(g;x_1,x_2) = I_\triangle(g;x_1,x_2)\circ I_\square(g;x_3,x_4).\]
\item $I_1(g;x_2,x_3)\circ I_2(g;x_1,x_2)\rightarrow I_1(g;x_1,x_3)$.
\item $I_2(g;x_2,x_3)\circ I_1(g;x_1,x_2)\rightarrow I_1(g;x_1,x_3)$.
\item $I_1(g;x_2,x_3)\circ I_1(g;x_1,x_2)\rightarrow I_2(g;x_1,x_3)$.
\item $I_2(g;x_2,x_3)\circ I_2(g;x_1,x_2)\rightarrow I_2(g;x_1,x_3)$.
\end{enumerate}
\end{lem}

\begin{rem} Notice that properties (iii-vi) show that there is a sort of ``parity'' on the intersection moves. Two moves of the same type act like $I_2$ when composed, and two moves of different types act like $I_1$ when composed.

Despite these reduction results, there is no corresponding ``factoring'' result in general. That is, we cannot necessarily split a single intersection move into a pair of moves.
\end{rem}

\begin{proof} 
The proofs of (i) and (ii) are immediate. For (i), no intersection move affects more than one element of $G$. As the condition of being $g$-annihilating and $g$-unequal only depends on $g$ (and no other element of $G$), the application of $I_\square(g)$ or $I_\triangle(g')$ does not prevent the other from being applied. Similarly, for (ii), as the pair of intersection moves affect distinct pairs of elements of $I$, one intersection move does not influence the application of the other.

To prove (iii)-(vi), fix $g\in G_j$ and $x_1,x_2,x_3\in I$. In all cases, the values $b(g,x_1)$ and $b(g,x_3)$ switch once, and the value $b(g,x_2)$ switches twice. Thus the net effect of the two intersection moves produces a map $b'$ such that $b'(g,x_2)=b(g,x_2)$, $b'(g,x_1) = b(s_j,x_1)-b(g,x_1), \text{ and } b'(g,x_3) = b(s_j,x_3)-b(g,x_3)$. This precisely matches the effect of an intersection move applied to $(g;x_1,x_3)$. It remains to be shown that the appropriate move {\em can} be applied. We consider this below for case (iii), but the remaining cases are handled similarly.

In case (iii), $x_1$ and $x_2$ start $g$-unequal. After $I_2(g;x_1,x_2)$ is applied, $x_2$ and $x_3$ must become $g$-annihilating for $I_1(g;x_2,x_3)$ to be applied. Hence $x_1$ and $x_3$ were $g$-annihilating initially. So we may apply $I_1(g;x_1,x_3)$ as desired.
\end{proof}

Given a $G_i$-complementary pair, we now show that intersection moves can always be applied in pairs to preserve complementarity.

\begin{lem} \label{lem:comp} Let $g_1,g_2$ be a $G_i$-complementary pair in $T$. If a sequence of intersection moves is applied to $g_1$, then there is a sequence of intersection moves applied to $g_2$ such that, after the sequence is applied, $g_1$ and $g_2$ are still $G_i$-complementary.\end{lem}

\begin{proof} Let $N=I_{n_k}(g_1;x_k,x_k')\circ\cdots\circ I_{n_1}(g_1;x_1,x_1')$ be applied to $T$ where $n_j\in \set{1,2}$ for each $j$. Call the result $T'$. We claim that the sequence $N'=I_{n_k}(g_2;x_k,x_k')\circ\cdots\circ I_{n_1}(g_2;x_1,x_1')$ can be applied to $T'$. By Lemma~\ref{lem:props}(i), it suffices to consider the case where $k=1$. That is, if we can match the original sequence at the first step, then we can commute the matching intersection move past the rest of $N$ and consider a shorter sequence.

Assume $N=I_1(g_1;x,x')$. Then $x,x'$ are $g_1$-annihilating. In particular, we have $\boxed{b(s_i,x)+b(s_i,x')=0}$. As $g_1,g_2$ are $G_i$-complementary in $T$, \[b(g_1,x) + b(g_2,x) = b(s_i,x) \text{ and } b(g_1,x') + b(g_2,x') = b(s_i,x').\] Thus
\[b(g_2,x) + b(g_2,x') = b(g_1,x) + b(g_2,x) + b(g_1,x') + b(g_2,x') = b(s_i,x) + b(s_i,x') = 0.\]
Hence $x,x'$ are $g_2$-annihilating in $T$ (and thus $T'$) and so $I_1(g_2;x,x')$ can be applied.

Assume $N=I_2(g_1;x,x')$. Then $x,x'$ are $g_1$-unequal in $T$ and so \[b(s_i,x) = b(s_i,x') \text{ and } b(g_1,x)\neq b(g_1,x').\] Thus
\[b(g_1,x)+ b(g_2,x) = b(s_i,x) = b(s_i,x') = b(g_1,x')+b(g_2,x').\]
Hence we obtain $b(g_2,x) = b(g_1,x') \neq b(g_1,x) = b(g_2,x')$. That is, $x,x'$ are $g_2$-unequal in $T$ (and $T'$) and so we may apply $I_2(g_2;x,x')$.
\end{proof}

\begin{prop} \label{prop:reduced} Any finite sequence of intersection moves $N$ on $T$ can be replaced by a sequence of intersection moves $N_k\circ \cdots \circ N_1$ such that:
\begin{itemize}[leftmargin=*]
\item Each intersection move in $N_k$ acts on an element $g_k\in G$ with $g_i\neq g_j$ for $i\neq j$.
\item For $x\in I$, there is at most one move of the form $I_\square(g_i;x,\ul{~~})$ in $N_i$.
\end{itemize}
Consequently, each $N_i$ consists of at most $\flr{\frac{\# I}{2}}$ intersection moves.
\end{prop}

\begin{proof} By Lemma~\ref{lem:props}(i), we immediately obtain a decomposition $N = N_k\circ \cdots \circ N_1$. It remains to show that, within $N_i$, we can assume there is at most one move of the form $I_\square(g_i;x,\ul{~~})$. Suppose there were two moves acting on $g_i$ at $x$, $I_\square(g;x,x')$ and $I_\triangle(g;x,x'')$. By Lemma~\ref{lem:props}(ii), we may assume that they occur consecutively. Then either $x''=x'$ or $x''\neq x'$.

If $x''=x'$, then $\triangle=\square$ (as a pair cannot be $g$-annihilating and $g$-unequal) and the two moves cancel. If $x''\neq x'$, then the pair of moves can be replaced by a single intersection move by Lemma~\ref{lem:props}(iii-vi). Thus we may replace pairs of moves in $N_i$ until there is no more than one occurrence of $I_\square(g_i;x,\ul{~~})$.

Finally, as each intersection move applies to pairs of elements from $I$, there can be at most one move per pair of distinct elements. Thus $N_i$ consists of at most $\flr{\frac{\# I}{2}}$ intersection moves.\end{proof}

We consider a sequence of intersection moves as in Proposition~\ref{prop:reduced} to be {\em reduced}. Such sequences simplify arguments considerably.

The next two lemmas provide technical results that dramatically simplify arguments in the proof of Proposition~\ref{prop:main}. The first result generalizes Lemma~\ref{lem:props}, giving conditions which allow a pair of intersection moves, acting on the same element $g$, to be replaced by a different pair of intersection moves. The second result is similar, but involves a $G_k$-complementary pair.

\begin{lem}\label{lem:tech1}
Let $g\in G_k$ and $x_1,x_2\in I$ such that $b(s_k,x_2)=-b(s_k,x_1)\neq 0$. Suppose $x_1,x_2$ are $g$-annihilating and we apply $N =I_\triangle(g;x_2,x_2')\circ I_\square(g;x_1,x_1')$ to $T$ with $x_1'\neq x_2'$. Then
\[N\rightarrow I_1(g;x_1,x_2)\circ I_\ocircle(g;x_1',x_2') \text{ where } \ocircle = \begin{cases} 1 & b(s_k,x_1')\neq b(s_k,x_2')\\ 2 & b(s_k,x_1') = b(s_k,x_2')\end{cases}.\]
\end{lem}

\begin{rem} If $x_1'=x_2'$, then we obtain the results in Lemma~\ref{lem:props}(iii-vi).\end{rem}

\begin{proof} After $N$ is applied to $T$, all values $b(g,x_j),b(g,x_j')$ are switched. Since $x_1,x_2$ are $g$-annihilating on $T$, we know that $I_1(g;x_1,x_2)$ could be applied to $T$. In particular, $\boxed{b(g,x_1) + b(g,x_2)=0}(\ast)$.

It suffices to show that $I_\ocircle$ applies. Since $N$ {\em can} be applied to $T$, there are restrictions on the values $b(s_k,\ul{~~})$ by definition of $I_\square,I_\triangle$. That is, \[b(s_k,x_1') = \begin{cases} b(s_k,x_1) &\text{if } \square =2\\ b(s_k,x_2) &\text{if } \square =1\end{cases}\text{ and } b(s_k,x_2') = \begin{cases} b(s_k,x_2) & \text{if }\triangle =2\\ b(s_k,x_1) &\text{if } \triangle =1\end{cases}.\]
This shows that $b(s_k,x_1')\neq 0$ and $b(s_k,x_2')\neq 0$.

Suppose $\triangle=\square$. Then either the pairs $x_1,x_1'$ and $x_2,x_2'$ are both $g$-unequal or both $g$-annihilating. Since $x_1,x_2$ are $g$-annihilating, $\boxed{b(s_k,x_1') \neq b(s_k,x_2')}$. Thus $x_1',x_2'$ are $g$-annihilating.

%

Suppose $\triangle\neq\square$. Up to relabeling and reordering by Lemma~\ref{lem:props}(ii), it suffices to consider the specific case $\square=1,\triangle=2$. Thus we have $\boxed{b(s_k,x_1')=b(s_k,x_2')}$. It suffices to show $b(g,x_1')\neq b(g,x_2')$.

Since we are applying $I_\square(g;x_1,x_1'),I_\triangle(g;x_2,x_2')$, we know that
\[b(g,x_1) + b(g,x_1') = 0 \text{ and } b(g,x_2)\neq b(g,x_2').\]
Thus
\[b(g,x_1) + b(g,x_1') = 0 \overset{(\ast)}{=} b(g,x_1) + b(g,x_2).\]
Hence $b(g,x_1') = b(g,x_2) \neq b(g,x_2')$, and so we may apply $I_2(g;x_1',x_2')$.
\end{proof}

\begin{lem}\label{lem:tech2}
Assume $g_1,g_2\in G_k$ are $G_k$-complementary and we apply $N=I_\triangle(g_2;x',y')\circ I_\square(g_1;x,y)$ to $T$ where $y\neq x'$ and $y'\neq x$. Let $T'$ be the resulting based matrix. Then we may replace $g_1,g_2$ with a different $G_k$-complementary pair $g_1',g_2'$ and instead obtain $T'$ by applying $I_1(g_1';x,x')$ and/or $I_\ocircle(g_1;y,y')$.
\end{lem}


\begin{proof} First assume $y'=y$ and $x'\neq x$. Then we are applying $I_\triangle(g_2;x',y)\circ I_\square(g_1;x,y)$ to $T$. Diagrammatically, the weaving map for $T'$ differs from the weaving map for $T$ by the marked entries in \begin{tabular}{c|ccc}
& $x$ & $x'$ & $y$\\\hline
$g_1$ & $\ast$ & & $\ast$\\
$g_2$ & & $\ast$ & $\ast$ 
\end{tabular}. Since $g_1,g_2$ are complementary in $T$, and we changed the weaving map at both $(g_1,y),(g_2,y)$, they are complementary at $y$ in $T'$. Then we also obtain $T'$ by applying $I_1(g_1';x,x')\circ M_{3,k}(g_1',g_2')\circ M_{3,k}^{-1}(g_1,g_2)$.

To be more concrete, applying $M_{3,k}(g_1',g_2')$ to $M_{3,k}^{-1}(g_1,g_2)(T)$ adds elements $g_1',g_2'$ to obtain a woven based matrix $T''$ where $g_1',g_2'$ agree with $g_1,g_2$ except as follows:
\[\p''(g_1',x') = \p(g_2,x'),\qquad \p''(g_2',x')=\p(g_1,x')\]
and 
\[\p''(g_1',y') = \p(g_2,y'),\qquad \p''(g_2',y') = \p(g_1,y').\]
Applying $I_1(g_1';x,x')$ to $T''$, and relabeling $g_1'$ as $g_1$ and $g_2'$ as $g_2$, recovers $T'$.

Now suppose $y'\neq y$ and $x'\neq x$. Then the weaving map for $T'$ differs from that for $T$ at the marked entries in \begin{tabular}{c|cccc}
& $x$ & $x'$ & $y$ & $y'$\\\hline
$g_1$ & $\ast$ & & $\ast$ &\\
$g_2$ & & $\ast$ & & $\ast$ 
\end{tabular}. This same result is obtained by \[I_1(g_1';x,x')\circ I_\ocircle(g_1';y,y')\circ M_{3,k}(g_1',g_2')\circ M_{3,k}^{-1}(g_1,g_2).\] As before, we replace $g_1,g_2$ with $g_1',g_2'$ which agree except at $x',y'$. 
Finally, notice that $\ocircle = \begin{cases} 1 & \text{if } b(s_k,y') = -b(s_k,y)\\ 2 & \text{if } b(s_k,y') = b(s_k,y)\end{cases}$ and that applying the given intersection moves to $T''$ gives us $T'$.

The cases with $x'=x$ differ only by the fact that there is no need to apply $I_1(g_1;x,x')$.
\end{proof}

\section{Primitive Based Matrices}

In this section, we define a ``minimal'' representative for the homology class of a woven based matrix, show it always exists, and prove uniqueness in a certain sense. 
A woven based matrix $T_\bullet = (G_{i\bullet},I_\bullet,s_{i\bullet},b_\bullet)$ is {\em primitive} if, after any sequence of intersection moves, no $M_j^{-1}$ moves may be applied to $T_\bullet$. 


\begin{thm} \label{thm:main} Let $T$ be a woven based matrix. Then there is a primitive woven based matrix $T_\bullet$ such that $T$ is homologous to $T_\bullet$ and, moreover, $T_\bullet$ is unique up to isomorphism and a finite sequence of intersection moves.
\end{thm}

To prove Theorem~\ref{thm:main}, we rely on a series of technical results. The proof itself involves taking homologous woven based matrices and showing that the sequence of moves relating them can be assumed to be of a given form.


For the remainder of the section, let $T = (G_1,\dots,G_n,I,s_1,\dots,s_n,b)$ be a woven based matrix with weaving map $\p$ and let $G=G_1\cup\dots\cup G_n$. 

\begin{prop} \label{prop:main} Any sequence of moves $M_j^{-1}\circ N\circ M_i$ on $T$, where $N$ consists of intersection moves, can be replaced by a sequence of the form $N_E\circ E\circ \Phi\circ E'\circ N_{E'}$ such that:
\begin{itemize}[leftmargin=*]
\item $N_E,N_{E'}$ consist of intersection moves,
\item $\Phi$ is an isomorphism of woven based matrices,
\item $E$ consists of at most one elementary extension,
\item $E'$ consists of at most one inverse elementary extension.
\end{itemize}
\end{prop}

We rely heavily on the following two lemmas to systemically reduce the proof of Proposition~\ref{prop:main} to specific cases.

\begin{lem} \label{lem:empty} The composition $M_j^{-1}\circ M_i$ can be replaced by sequences in which all inverse extensions apply before extensions, i.e., by a sequence of the form $M_k^{\pm1}$, $M_k\circ M_\ell^{-1}$, or $\Phi$ where $\Phi$ is an isomorphism.
\end{lem}

\begin{proof} If $i,j\neq 4$, then the result follows directly from Lemma 6.1.1 in~\cite{Turaev}. Moreover, $M_1^{\pm 1}$, $M_2^{\pm 1}$, and $M_3^{\pm 1}$ commute when applied to different sets $G_m$, and so it holds trivially in this case.

Thus we need only consider the cases where $i=4$ or $j=4$. For notational simplicity, we consider $M_4^{-1}\circ M_i$ and $M_i^{-1}\circ M_4$ for $i=1,2,3,4$.

Suppose $i\in \set{1,2,3}$. Then $M_i$ is adding elements to some $G_m$ while $M_4^{-1}$ removes elements of $I$. Since both may be applied, and involve disjoint sets, the net effect is the same if they are switched. Hence $M_4^{-1}\circ M_i\rightarrow M_i\circ M_4^{-1}$. The same argument applies to $M_i^{-1}\circ M_4$.

Finally consider $i=4$. Write $M_4^{-1}\circ M_4 = M_4^{-1}(y_1,y_2)\circ M_4(x_1,x_2)$. Then we have three cases:
\begin{enumerate}[1.,leftmargin=*]
\item If $y_1=x_1$ and $y_2=x_2$, then $M_4^{-1}(y_1,y_2)\circ M_4(x_1,x_2)\rightarrow\Id$.
\item If $y_1\neq x_1$ and $y_2\neq x_2$, then the extensions commute. That is, adding $x_1,x_2$ to $I$ does not affect whether $y_1,y_2$ are sum-annihilating. Since $M_4^{-1}(y_1,y_2)$ is being applied, we know that $y_1,y_2$ are sum-annihilating.
\item If $x_1=x_1'$ and $x_2\neq x_2'$ then, as $(x_1,x_2)$ and $(x_1,y_2)$ are both sum-annihilating pairs, it must be that $x_2$ takes the same values at $y_2$. That is, $b(a,x_2)=b(a,y_2)$ for all $a$. Hence $M_4^{-1}\circ M_4 \rightarrow\Phi$ where $\Phi$ is the isomorphism relabeling $y_2$ as $x_2$.
\end{enumerate}
In all cases, we have replaced $M_j^{-1}\circ M_i$ by a sequence of the desired form.
\end{proof}

The second reduction lemma is a consequence of Proposition~\ref{prop:reduced}.

\begin{lem} \label{lem:N}
In Proposition~\ref{prop:main}, we may assume that $N$ is reduced and consists only of intersection moves affecting both elements added by $M_i$ and elements removed by $M_j^{-1}$. 

Furthermore, if $i=4$ or $j=4$ (but not both), then $N$ consists of at most 2 intersection moves per non-intersection element added (removed) by $M_i$ ($M_j^{-1}$). If $i=j=4$, $N$  has at most two moves affecting each $g\in G$.
\end{lem}

\begin{proof} Consider $M_j^{-1}\circ N\circ M_i$. By Proposition~\ref{prop:reduced}, we may assume that $N$ is reduced. Thus, by Lemma~\ref{lem:props}(i,ii), all intersection moves in $N$ commute with each other. 

Now consider $I_\square(g;x,x')\circ M_i$. If none of $g,x,x'$ are added by $M_i$, then $I_\square(g;x,x')$ could be applied before $M_i$. Similarly, $M_j^{-1}\circ I_\square(g;x,x')$ can be replaced by $I_\square(g;x,x')\circ M_j^{-1}$ so long as none of $g,x,x'$ are removed by $M_j^{-1}$. Notice that, unless $I_\square(g;x,x')$ affects elements both added by $M_i$ and removed by $M_j^{-1}$, then it commutes past one of the two operations.

Thus we may replace $M_j^{-1}\circ N\circ M_i$ by $N_3\circ M_j^{-1}\circ N_2\circ M_i\circ N_1$ where $N_1,N_2,N_3$ are reduced sequences of intersection moves and $N_2$ consists only of intersection moves affecting elements added by $M_i$ and removed by $M_j^{-1}$.

In Proposition~\ref{prop:main}, we allow for extra sequences of intersection moves $N_{E'}$ and $N_{E}$. Thus we may lump $N_1,N_3$ in with the respective sequence. So we may assume $N=N_2$ for Proposition~\ref{prop:main}.

It remains to show that we may bound the number of intersection moves in the reduced sequence $N_2$ in the case where $i=4$ or $j=4$.

Suppose $i\neq 4$ and $j=4$ and $M_j^{-1}$ removes $x,x'$. Let $g$ be an element added by $M_i$. By the construction of $N_2$, an intersection move affecting $g$ has the form $I_\square(g;x,\ul{~~})$ or $I_\square(g;x',\ul{~~})$. Since $N_2$ is reduced, at most one of each of these can occur. That is, at most two intersection moves in $N_2$ apply to $g$. Since $M_i$ adds one or two elements, $N_2$ consists of at most $4$ intersection moves. An analogous argument holds for $i=4$, $j\neq 4$.

Finally, suppose $i=j=4$. Then each intersection move in $N_2$ applies to $(g;x,\ul{~~})$ where $x$ was added by $M_i$ and one of the two intersection elements must be removed by $M_j^{-1}$. Since $M_i,M_j^{-1}$ each involve two elements of $I$ and $N_2$ is reduced, there are at most two such disjoint pairs. Hence at most $2$ moves in $N_2$ can apply to $g$.
\end{proof}

We now turn our attention to the proof of Proposition~\ref{prop:main}.

\begin{proof}[Proof of Proposition~\ref{prop:main}.]
There are a total of 16 cases to consider. By Lemma~\ref{lem:N}, we need only consider $N$ consisting of intersection moves affecting elements both added by $M_i$ and removed by $M_j^{-1}$. If $N=\Id$, then the result follows by Lemma~\ref{lem:empty}. So we assume $\boxed{N\neq \Id}$.

Notice that the form of $N$ implies that there must be overlap between the elements added to $G\cup I$ and those removed from $G\cup I$. That is, $I_1,I_2$ can only act on a single element of $G$ and, unless there is overlap, any intersection move would commute past either $M_i$ or $M_j^{-1}$ (implying $N=\Id$). So it is enough to consider these cases where there is overlap. 

We start with the cases where $i,j\neq 4$. In particular, by overlap, $M_i$ and $M_j^{-1}$ both apply to a fixed set $G_k$.

\begin{enumerate}[leftmargin=*]
\item[(a)] $\boxed{i,j\in\set{1,2}}$: Consider $M_i(g)$ and $M_j^{-1}(h)$. By overlap, $g=h$. As each intersection move in $N$ applies to $(g;b,b')$, $M_j^{-1}(g)\circ N\circ M_i(g)\rightarrow \Id$.

\item[(b)] $\boxed{i\in\set{1,2},j=3}$: Consider $M_i(g)$ and $M_3^{-1}(h_1,h_2)$. By overlap, $g=h_1$ without loss of generality. Hence each intersection move in $N$ affects $g=h_1$ and transforms $g$ into a complement for $h_2$. At this stage, applying $N^{-1}=N$ again would turn $g$ back into an annihilating ($i=1$) or core ($i=2$) element. By Lemma~\ref{lem:comp}, as core and annihilating elements are complementary, there is a sequence $N'$ which would thus turn $h_2$ into a core ($i=1$) or annihilating ($i=2$) element. Thus the net effect of $M_3^{-1}\circ N\circ M_i$ is removing the element $h_2$ which, after $N'$ is applied, is revealed to be a core or annihilating element. That is,
\[\begin{tikzcd}[row sep=.5ex]
M_3^{-1}\circ N\circ M_i \ar{r}{i=1}\ar[swap]{dr}{i=2} & M_2^{-1}\circ N'\\
&  M_1^{-1}\circ N'
\end{tikzcd}.\]

\item[(c)] $\boxed{i=3,j\in\set{1,2}}$: This case follows analogously to Case (b), except that the sequence is replaced by either $N'\circ M_1$ or $N'\circ M_2$.

\item[(d)] $\boxed{i=j=3}$: Consider $M_3^{-1}(h_1,h_2)$ and $M_3(g_1,g_2)$. By overlap, we have two cases without loss of generality.
\begin{itemize}[leftmargin=*]
\item If $g_1=h_1$ and $g_2\neq h_2$, then $N$ consists of intersection moves acting on $g_1=h_1$ by Lemma~\ref{lem:N}. Thus $g_1,g_2$ start out complementary and, after $N$, $g_1,h_2$ are complementary. By Lemma~\ref{lem:comp}, there is a sequence of intersection moves $N'$ such that $N'$ transforms $g_2$ into a complementary element to $g_1$ (after $N$ is applied). That is, $N'$ transforms $g_2$ into a copy of $h_2$. Thus $M_3^{-1}(g_1,h_2)\circ N\circ M_3(g_1,g_2) \rightarrow \Phi\circ N'$ where $\Phi$ is the isomorphism relabeling $h_2$ as $g_2$.

\item If $g_1=h_1$ and $g_2=h_2$, then every move in $N$ affects $g_1$ or $g_2$ by Lemma~\ref{lem:N}. 
Thus $M_3^{-1}(g_1,g_2)\circ N\circ M_3(g_1,g_2)\rightarrow \Id$.
\end{itemize}
Hence $M_3^{-1}\circ N\circ M_3$ can be replaced by intersection moves and an isomorphism.
\end{enumerate}

Cases (a)-(d) handle the first 9 cases. All remaining cases involve the addition or removal of a sum-annihilating pair, and the full result of Lemma~\ref{lem:N}.

\begin{enumerate}[leftmargin=*]
\item[(e)] $\boxed{i\in\set{1,2},j=4}$: Consider $M_4^{-1}(x_1,x_2)\circ N\circ M_1(g)$. By Lemma~\ref{lem:N}, we know $N$ consists of intersection moves of the form $I_\square(g;x_1,\ul{~~})$ or $I_\square(g;x_2,\ul{~~})$. As $N\neq \Id$ and is reduced, it consists of either 1 or 2 such moves and only one move can apply to each of $x_1,x_2$. Since $g$ is annihilating or core, no $I_2(g)$ moves can be applied. Finally, to ensure that $x_1,x_2$ end sum-annihilating (to be removed by $M_4^{-1}$), we have $N = I_1(g;x_1,x_2)$ or $N=I_1(g;x_2,x_2')\circ I_1(g;x_1,x_1')$.

In the first case, \[M_4^{-1}(x_1,x_2)\circ I_1(g;x_1,x_2)\circ M_i(g) \rightarrow M_i(g)\circ M_4^{-1}(x_1,x_2).\] In the second case, by Lemma~\ref{lem:tech1}, \[M_4^{-1}(x_1,x_2)\circ N\circ M_i(g) \rightarrow I_1(g;x_1',x_2')\circ M_i(g)\circ M_4^{-1}(x_1,x_2).\] In either case, we obtain compositions of the desired form.

\item[(f)] $\boxed{i=4,j\in\set{1,2}}$: Similar to Case (e), with $I_2(g)$ moves being ruled out because $g$ must become either a core or annihilating element (and $I_2(g)$ prevents this possibility). 

\item[(g)] $\boxed{i=3,j=4}$: Consider $M_3(g_1,g_2)$ and $M_4^{-1}(x_1,x_2)$. In what follows, we use a ``fix it as we go'' method. That is, having added a complementary pair, the intersection moves in $N$ must alter the weaving map at $g_1,g_2$ until $x_1,x_2$ are sum-annihilating (and subsequently removed). We say that $g$ is {\em fixed} if $x_1,x_2$ are $g$-annihilating; {\em broken} if $x_1,x_2$ are not $g$-annihilating. We begin with some comments:
\begin{itemize}[leftmargin=*]
\item $(\ast)$ Since $g_1,g_2$ are complementary, $g_1$ and $g_2$ are either both fixed or both broken.

\item $(\ast\ast)$ If $g_1$ (respectively $g_2$) is broken, then exactly one $I_\square(g_1)$ move can be applied in $N$. Since the goal is for $x_1,x_2$ to be sum-annihilating everywhere (particularly at $g_1$), such a move must be applied; $g_1$ must be fixed and applying $I_\square(g_1)$ will fix it. However, since no more than two moves can be applied to $g_1$ by Lemma~\ref{lem:N}, a second move would break $g_1$ permanently.

\item $(\dagger)$ If $g_1$ is fixed, then it is enough to understand the moves on $g_1$ independently of $g_2$ (and vice versa). That is, the ``replacement'' sequences will be independent. This helps reduce the number of cases that we need to consider significantly.
\end{itemize}
The two overarching cases depend on whether $g_1$ and $g_2$ are broken or fixed.
\begin{itemize}[leftmargin=*]
\item $\boxed{g_1,g_2 \text{ fixed}}$: It suffices to consider only $g_1$ by $(\dagger)$. Since $g_1$ is already fixed, we either need to apply an intersection move $I_1(g_1;x_1,x_2)$ (since $x_1,x_2$ are $g_1$-annihilating, we cannot apply $I_2$) or a pair of intersection moves $I_\triangle(g_1;x_2,x_2')\circ I_\square(g_1;x_1,x_1')$.

In the former case, replace \[M_4^{-1}(x_1,x_2)\circ I_1(g_1;x_1,x_2)\circ M_3(g_1,g_2)\rightarrow M_4^{-1}(x_1,x_2)\circ M_3(g_1,g_2).\] 

Now consider $I_\triangle(g_1;x_2,x_2')\circ I_\square(g_1;x_1,x_1')$. By Lemma~\ref{lem:tech1}, 
\begin{align*}
&M_4^{-1}(x_1,x_2)\circ (I_\triangle(g_1;x_2,x_2')\circ I_\square(g_1;x_1,x_1'))\circ M_3(g_1,g_2)\\
&\rightarrow M_4^{-1}(x_1,x_2)\circ I_\ocircle(g_1;x_1',x_2') \circ M_3(g_1,g_2)\\ &\quad= I_\ocircle(g_1;x_1',x_2')\circ M_4^{-1}(x_1,x_2)\circ M_3(g_1,g_2).\end{align*}
Thus we may assume $N=\Id$ and we finish by Lemma~\ref{lem:empty}.

\item $\boxed{g_1,g_2 \text{ broken}}$: By $(\ast\ast)$, we apply exactly two intersection moves, one to each of $g_1,g_2$. By Lemma~\ref{lem:props}(i), these may occur in either order. Note that, since $g_1,g_2$ are broken, we cannot apply $I_1(g_\ell;x_1,x_2)$ (nor $I_2(g_\ell;x_1,x_2)$). 
By Lemma~\ref{lem:tech2}, we may replace
\begin{align*}
&M_4^{-1}(x_1,x_2)\circ (I_\triangle(g_2;x_2,y')\circ I_\square(g_1;x_1,y))\circ M_3(g_1,g_2)\\
&\rightarrow M_4^{-1}(x_1,x_2)\circ I_1(g_1';x_1,x_2)\circ I_\ocircle(g_1';y,y')\circ M_3(g_1',g_2').\end{align*}
Now $I_\ocircle(g_1';y,y')$ commutes out and $g_1',g_2'$ are both fixed. Hence we are in the previous case of fixed elements.


\end{itemize}

\item[(h)] $\boxed{i=4,j=3}:$ 
Consider $M_3^{-1}(g_1,g_2)\circ N\circ M_4(x_1,x_2)$. We call $x\in I$ {\em fixed} at $g_1,g_2$ if $b(g_1,x) + b(g_2,x)= b(s_k,x)$, and {\em broken} otherwise. Since the goal is to remove a $G_k$-complementary pair $g_1,g_2$, $N$ must fix $x_1,x_2$ (and possibly other elements of $I$). We begin with two remarks:

\begin{itemize}[leftmargin=*]
\item $(\ast)$ Since $x_1,x_2$ are sum-annihilating, either $x_1,x_2$ are initially both fixed or they are both broken. Indeed, suppose $x_1$ is fixed. Then, as they are sum-annihilating, we have
\[b(g_1,x_1) + b(g_1,x_2) = 0 \text{ and } b(g_2,x_1)+b(g_2,x_2) = 0.\]
Hence
\[b(g_1,x_2) + b(g_2,x_2) = -(b(g_1,x_1) + b(g_2,x_1)) = -b(s_k,x_1) = b(s_k,x_2).\]
Thus $x_2$ is fixed at $g_1,g_2$.

\item $(\ast\ast)$ If $x_1$ (respectively $x_2$) is broken, then exactly one $I_\square(\ul{~~};x_1,\ul{~~})$ move can be applied in $N$. By Lemma~\ref{lem:N}, at most two intersection moves in $N$ will involve $x_1$ (one for each of $g_1,g_2$). Since $x_1$ is broken, one move must be applied to $x_1$ to fix it, but a second move would break $x_1$ permanently. 

\end{itemize}
The two overarching cases depend on whether $x_1,x_2$ are fixed or broken.

\begin{itemize}[leftmargin=*]
\item $\boxed{x_1,x_2 \text{ fixed}}$: 
If the moves in $N$ affecting $x_1$ do not affect $x_2$, then we may consider the intersection moves independently. Since $x_1$ is fixed, there are either 0 or 2 moves applied to $x_1$. Then the contribution of $N$ to $x_1$ may be $I_\triangle(g_2;x_1,x')\circ I_\square(g_1;x_1,x)$.

\begin{enumerate}[1.,leftmargin=*] 
\item Supposing $x=x_2$ and $x'=x_2$ forces $\square=\triangle=1$. So 
\[M_3^{-1}\circ N\circ M_4\rightarrow M_3^{-1}\circ M_4 = M_4\circ M_3^{-1}\]
by Lemma~\ref{lem:empty}.

\item Assuming $x=x_2$ or $x'=x_2$ (but not both), $N$ consists of exactly 3 intersection moves. Up to relabeling, we may assume $x=x_2$ and $x'\neq x_2$. Then
\[N = I_\square(g_2;x_2,x'')\circ I_\triangle(g_2;x_1,x')\circ I_1(g_1;x_1,x_2)\]
where $x''\neq x'$ (since $N$ is reduced). Since $x_1,x_2$ are sum-annihilating in $T$ and thus $g_2$-annihilating, we apply Lemma~\ref{lem:tech1} and replace
\[N \rightarrow I_1(g_2;x_1,x_2)\circ I_\ocircle(g_2;x',x'')\circ I_1(g_1;x_1,x_2).\]
Hence
\begin{align*}
M_3^{-1}\circ N\circ M_4 &\rightarrow M_3^{-1}\circ I_\ocircle(g_2;x',x'')\circ M_4\\ &\quad= M_3^{-1}\circ M_4\circ I_\ocircle(g_2;x',x'')\end{align*}
and we apply Lemma~\ref{lem:empty}.

\item Assume $x\neq x_2$ and $x'\neq x_2$. Then we may treat $x_1,x_2$ independently. Suppose $N=I_\triangle(g_2;x_1,x')\circ I_\square(g_1;x_1,x)$. Then $x,x'$ must both be fixed by this sequence as well. So $x,x'$ are both broken and we replace
\[M_3^{-1}\circ N\circ M_4 \rightarrow M_3^{-1}\circ M_4\circ I_\ocircle(g_1;x,x')\]
where $\ocircle = \begin{cases} 1 & \text{if } b(x,s_k) \neq b(x',s_k)\\ 2 & \text{if } b(x,s_k)=b(x',s_k)\end{cases}$ by Lemma~\ref{lem:tech1}. Again we apply Lemma~\ref{lem:empty}.
\end{enumerate}

\item $\boxed{x_1,x_2 \text{ broken}}$: By $(\ast\ast)$, we need only apply a single move to each of $x_1,x_2$. There are 3 distinct families of possibilities:

\begin{enumerate}[1.,leftmargin=*]
\item $N = I_1(g_\ell;x_1,x_2)$ for $\ell\in\set{1,2}$.

Then we replace $M_3^{-1}\circ N\circ M_4(x_1,x_2) \rightarrow M_3^{-1}\circ M_4(x_1',x_2')$ where $x_1',x_2'$ agree with $x_1,x_2$ except at $g_\ell$ (at which they take on the alternate values) and finish by Lemma~\ref{lem:empty}. 

\item $N= I_\triangle(g_\ell;x_2,y_2)\circ I_\square(g_\ell;x_1,y_1)$ for $\ell\in\set{1,2}$.

Since $N$ is reduced, $y_1\neq y_2$.
By Lemma~\ref{lem:tech1},
\[M_3^{-1}\circ N\circ M_4 \rightarrow M_3^{-1}\circ I_1(g_\ell;x_1,x_2)\circ I_\ocircle(g_\ell;y_1,y_2)\circ M_4.\]
As $I_\ocircle(g_\ell;y_1,y_2)$ commutes out, we reduce to Case 1.

\item $N = I_\triangle(g_2;x_2,y_2)\circ I_\square(g_1;x_1,y_1)$.


If $y_1=y_2$, then $y_1$ must have been fixed at $g_1,g_2$ initially. 
Hence we replace $M_3^{-1}\circ N\circ M_4\rightarrow M_3^{-1}\circ I_1(g_1;x_1,x_2)\circ M_4$, i.e., reduce to Case 1.

If $y_1\neq y_2$, then we know $y_1,y_2$ must also be broken. Hence we replace
\[M_3^{-1}\circ N\circ M_4 \rightarrow M_3^{-1}\circ I_\triangle(g_1;x_2,y_2)\circ I_\square(g_1;x_1,y_1)\circ M_4.\]
That is, $I_\triangle(g_1;x_2,y_2)$ is also applicable and we choose to apply it instead of $I_\triangle(g_2;x_2,y_2)$. Now we are in Case 2.
\end{enumerate}
\end{itemize}

\item[(i)] $\boxed{i=j=4}$: Consider $M_4^{-1}(y_1,y_2)\circ N\circ M_4(x_1,x_2)$. By Lemma~\ref{lem:N}, there are at most two intersection moves in $N$ applied to $g$ for fixed $g\in G$. Denote these moves by $N_g$. We consider subcases depending on whether $y_1=x_1$ and $y_2=x_2$. 
It is enough to show that, in each subcase, all possible outcomes 
can be achieved by a single replacement.

\begin{itemize}[leftmargin=*]
\item $\boxed{y_1=x_1,y_2=x_2}$: In the terminology of Case (g), $g$ is fixed at $x_1,x_2$ and, after the intersection move(s), it must still be. Since $N$ is reduced and by Lemma~\ref{lem:N}, either $N_g = I_1(g;x_1,x_2)$ or $N_g = I_\triangle(g;x_2,x_2')\circ I_\square(g;x_1,x_1')$ with $x_1'\neq x_2'$ distinct from $x_1,x_2$.

If $N_g =I_1(g;x_1,x_2)$, then we replace $N_g\rightarrow \Id$ for such $g$. 

If $N_g = I_\triangle(g;x_2,x_2')\circ I_\square(g;x_1,x_1')$, then we apply Lemma~\ref{lem:tech1} to replace
\[I_\triangle(g;x_2,x_2')\circ I_\square(g;x_1,x_1') \longrightarrow I_1(g;x_1,x_2)\circ I_\ocircle(g;x_1',x_2').\]
Then $I_\ocircle(g;x_1',x_2')$ commutes out and we may again assume $N_g=\Id$.

\item $\boxed{y_1\neq x_1, y_2\neq x_2}$: We show that $N$ transforms $x_1,x_2$ to have the same values as $y_1,y_2$. So $N\rightarrow \Phi$ where $\Phi$ is an isomorphism relabeling $y_1,y_2$ as $x_1,x_2$.

Since $x_1,x_2$ are sum-annihilating, we know that $g$ is fixed at $x_1,x_2$. We consider whether $g$ is fixed or broken at $y_1,y_2$. If $g$ is fixed, then we must apply 0 or 2 intersection moves by Lemma~\ref{lem:N}; if $g$ is broken, then we may only apply a single intersection move to fix $g$ at $y_1,y_2$.

Suppose $g$ is fixed at $y_1,y_2$ and two moves are applied. Then, by overlap, the values at each element $x_1,x_2,y_1,y_2$ must change. As $x_1,x_2$ were sum-annihilating and $y_1,y_2$ were $g$-annihilating before the moves, they still will be afterward. Hence 
\[M_4^{-1}(y_1,y_2)\circ N\circ M_4(x_1,x_2)\rightarrow \Phi\circ I_1(g;y_1,y_2)\] where $\Phi$ relabels $y_1,y_2$ as $x_1,x_2$. 

Now assume $g$ is broken at $y_1,y_2$. Then $I_\square(g;x_j,y_\ell)$ forces to $x_j$ to match the original $y_j$ in value at $g$. So $M_4^{-1}(y_1,y_2)\circ N\circ M_4(x_1,x_2)\rightarrow \Phi$ where $\Phi$ relabels $y_1,y_2$ as $x_1,x_2$. 

\item $\boxed{y_1=x_1,y_2\neq x_2}$: Since $N$ is reduced and we have overlap, there are four possibilities for $N_g$: $I_\square(g;x_1,x), I_1(g;x_1,x_2), I_1(g;x_1,y_2), \text{ or } I_2(g;x_2,y_2)$.

For the last three possibilities, after $N_g$, the value of $b(g,x_2)$ agrees with the original value of $b(g,y_2)$. (For $N_g=I_1(g;x_1,y_2)$, this follows since $x_1,x_2$ were initially sum-annihilating.) Thus, for all of these cases, we replace
\[M_4^{-1}\circ N\circ M_4\rightarrow  \Phi\]
where $\Phi$ relabels $y_2$ as $x_2$.


Finally, consider $N_g = I_\square(g;x_1,x)$ where $x\neq x_2,y_2$. Then $g$ is broken at $x_1,y_2$. Since $g$ is fixed at $x_1,x_2$, $\boxed{b(g,x_2)\neq b(g,y_2)}$ in $M_4(x_1,x_2)(T)$.

If $\square =1$, then $b(g,y_2)=b(g,x)$ after $I_\square$. Hence we replace 
\[M_4^{-1}\circ N\circ M_4 \rightarrow \Phi\circ I_2(g;y_2,x).\]

If $\square =2$, then $b(g,x_1)\neq b(g,x)$ after $I_\square$. As $x_1,x_2$ are sum-annihilating initially, $I_1(g;x_2,x)$ may be applied afterward. So we replace
\[M_4^{-1}\circ N\circ M_4 \rightarrow I_1(g;y_2,x)\circ \Phi.\]
\end{itemize}
\end{enumerate}
Hence we replace $M_j^{-1}\circ N\circ M_i$ by a sequence with inverse extensions occurring prior to any elementary extensions. 
\end{proof}

Now we can proceed with the proof of Theorem~\ref{thm:main}.

\begin{proof}[Proof of Theorem~\ref{thm:main}.]
The existence of a primitive based matrix $T_\bullet$ is clear. By the finiteness of $G\cup I$, there are a finite number of distinct results obtained by sequences of intersections moves. Thus the process of applying inverse extensions and checking the resulting finite possibilities for additional inverse extensions eventually terminates.

Now suppose that $T_\bullet'$ is another primitive based matrix for $T$. We want to show that $T_\bullet$ and $T_\bullet'$ are related by intersection moves. Since $T_\bullet$ and $T_\bullet'$ are both homologous to $T$, they are related to each other a sequence of moves $M$. That is, $M$ transforms $T_\bullet$ into $T_\bullet'$.

By Lemma~\ref{lem:empty} and Proposition~\ref{prop:main}, we may assume that all inverse extensions are applied before the elementary extensions. Thus we assume $M = M_E \circ \Phi \circ M_I$ where $\Phi$ is an isomorphism, $M_E$ consists of intersection moves and $M_j$ operations, and $M_I$ consists of intersection moves and $M_j^{-1}$ operations.

Since $T_\bullet$ is primitive, $M_I$ consists only of intersection moves. Similarly, $M_E$ consists only of intersection moves. Thus $M$ is a sequence of intersection moves and an isomorphism. Hence primitive woven based matrices are unique up to a sequence of intersection moves and an isomorphism.
\end{proof}

\begin{rem} Similar to Turaev's remark in~\cite{Turaev}, Theorem~\ref{thm:main} shows that we may always assume that a primitive woven based matrix is chosen with $G_{i\bullet}\subseteq G_i$, $I_\bullet \subseteq I$, $s_{i\bullet} = s_i$, and $b_\bullet$ is the restriction of $b$ after a sequence of intersection moves on $T$.
\end{rem}

\subsection{Examples}

\begin{ex} Recall the family of 2-strings $\B=\B(p_1,q_1,p_2,q_2,r,s)$ from Figure~\ref{fig:betafamily}. 
Assume that $T(\A_1),T(\A_2)$ are primitive. We claim that the multistring based matrix $T(\B)$ is also primitive.

Let $x_1,\dots,x_r$ and $y_1,\dots,y_s$ be the $r$ and $s$ parallel intersection arrows in $\B$, respectively. Further let $g_1,\dots,g_{p_1}$ and $h_1,\dots,h_{q_1}$ denote the arrows of $\A_1$. Then the weaving map, restricted to $G_1\times I$, has the form:
\[\begin{tabular}{c|cc}
& $x_i$ & $y_j$\\\hline
$s_1$ & $1$ & $-1$ \\
$g_k$ & $0$ & $-1$ \\
$h_\ell$ & $1$ & $0$
\end{tabular}.\]
In particular, no $I_1$ or $I_2$ moves can be applied to $T(\B)$ and no element of $I=\arr_\cap(\B)$ is sum-annihilating. 
Thus, as $T(\A_1),T(\A_2)$ are primitive, $T(\B)$ must be as well.\end{ex}

\begin{ex} Consider the 2-string $\sigma$ in Figure~\ref{fig:MBMex2}. 
The multistring based matrix of $\sigma$ is
\[\begin{tabular}{c|cccccccc}
& $s_1$ & $s_2$ & $g_1$ & $g_2$ & $x_1$ & $x_2$ & $x_3$ & $x_4$\\\hline
$s_1$ & $0$ & $0$ &$-1$ & $1$ & $-1$ & $-1$ & $1$ & $1$\\
$s_2$ & $0$ & $0$ & $0$ & $0$ & $1$ & $1$ & $-1$ & $-1$\\
$g_1$ & $1$ & $0$ & $0$ & $0$ & $1$ & $0$ & $0$ & $0$\\
$g_2$ & $-1$ & $0$ & $0$ & $0$ & $1$ & $0$ & $-1$ & $-1$\\
$x_1$ & $1$ & $-1$ & $-1$ & $-1$ & $0$ & $0$ & $0$ & $0$\\
$x_2$ & $1$ & $-1$ & $0$ & $0$ & $0$ & $0$ & $0$ & $0$\\
$x_3$ & $-1$ & $1$ & $0$ & $1$ & $0$ & $0$ & $0$ & $0$\\
$x_4$ & $-1$ & $1$ & $0$ & $1$ & $0$ & $0$ & $0$ & $0$\\
\end{tabular}.\]
In particular, no $M_j^{-1}$ moves can be applied to $T(\sigma)$. However, after applying $I_2(g_1;x_1,x_2)$, we may simplify the multistring based matrix by applying \[M_4^{-1}(x_2,x_4)\circ M_4^{-1}(x_1,x_3)\circ M_3^{-1}(g_1,g_2).\] Hence the associated primitive based matrix $T_\bullet(\sigma)$ is trivial.

On the other hand, consider the virtual Goldman bracket $B(\sigma)$ (see~\cite{Henrich}). Let $\sigma_{i}$ denote the virtual string obtained from $\sigma$ by smoothing the arrow $x_i$. Then the resulting virtual strings are depicted in Figure~\ref{fig:smoothings} and
\[B(\sigma) = -[\sigma_1] - [\sigma_2] + [\sigma_3] + [\sigma_4].\]
By applying Type 1 moves, we see $[\sigma_3]=[\sigma_4]$ and that these virtual strings have at most 4 self-arrows. Furthermore, $T(\sigma_1)$ and $T(\sigma_2)$ are primitive with 5 self-arrows and thus $[\sigma_1]$ and $[\sigma_2]$ are distinct from $[\sigma_3]=[\sigma_4]$. Hence $B(\sigma)\neq0$ and so $\sigma$ is nontrivial.
\end{ex}

\begin{figure}
\begin{tikzpicture}[thick,->]
	\draw[decoration={markings,mark=at position 0.5 with {\arrow[scale=1.5]{>}}},
        postaction={decorate},xshift=-2cm] (0,0) circle[radius=1cm];
							
	\draw[decoration={markings,mark=at position 1 with {\arrow[scale=1.5]{>}}},
        postaction={decorate},xshift=2cm] (0,0) circle[radius=1cm];
	
	\draw ([xshift=2cm] 240:1cm) .. controls (1.2cm,-1.2cm) and (-1.2cm,-1.2cm) .. ([xshift=-2cm] 300:1cm) node[pos=.5,below]{$x_2$};
	\draw ([xshift=2cm] 120:1cm) .. controls (1.2cm,1.2cm) and (-1.2cm,1.2cm) .. ([xshift=-2cm] 60:1cm) node[pos=.5,above]{$x_1$};
	\draw ([xshift=-2cm] 10:1cm) -- ([xshift=2cm] 170:1cm) node[pos=.5,above]{$x_4$};
	\draw([xshift=-2cm] -10:1cm) -- ([xshift=2cm] 190:1cm) node[pos=.5,below]{$x_3$};
	
	\draw[xshift=2cm] (270:1cm) to[out=90,in=30] (210:1cm) node[auto,above right]{$g_2$};
	\draw[xshift=2cm] (90:1cm) to[out=270,in=-30] (150:1cm) node[auto,below right]{$g_1$};
\end{tikzpicture}

\caption{Nontrivial 2-string  $\sigma$ with trivial $T_\bullet(\sigma)$.}
\label{fig:MBMex2}
\end{figure}
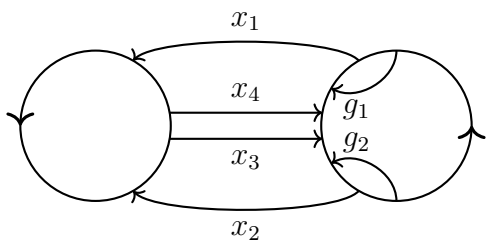

\begin{figure}
\begin{tikzpicture}[thick,->]
	\draw[decoration={markings,mark=at position 0.33 with {\arrow[scale=1.5]{>}}},
        postaction={decorate}] (0,0) circle[radius=1cm];
	\draw (90:1)--(270:1);
	\draw (0:1)--(180:1);
	\draw (20:1) to[out=200,in=160] (-20:1);
	\draw[<-] (290:1) to[out=120,in=70] (250:1);
	\draw[<-] (310:1) to [out=140,in=50] (230:1);
	
	\node at (0,-1.5) {$\sigma_1$};
							
\begin{scope}[xshift=3cm]	
	\draw[decoration={markings,mark=at position .33 with {\arrow[scale=1.5]{>}}},
        postaction={decorate}] (0,0) circle[radius=1cm];
	\draw (90:1)--(270:1);
	\draw[<-] (0:1)--(180:1);
	\draw (160:1) to[out=-20,in=20] (200:1);
	\draw (290:1) to[out=120,in=70] (250:1);
	\draw (310:1) to [out=140,in=50] (230:1);
	
	\node at (0,-1.5) {$\sigma_2$};
\end{scope}

\begin{scope}[xshift=6cm]	
	\draw[decoration={markings,mark=at position 1 with {\arrow[scale=1.5]{>}}},
        postaction={decorate}] (0,0) circle[radius=1cm];
	\draw (20:1) to[out=200,in=255] (75:1);
	\draw (40:1) to[out=220,in=-30] (150:1);
	\draw (125:1) to[out=305,in=270] (90:1);
	
	\draw (-20:1) to[out=-200,in=-255] (-75:1);
	\draw (-40:1) to[out=-220,in=30] (-150:1);
	
	\node at (0,-1.5) {$\sigma_3$};
\end{scope}

\begin{scope}[xshift=9cm]							
	\draw[decoration={markings,mark=at position 1 with {\arrow[scale=1.5]{>}}},
        postaction={decorate}] (0,0) circle[radius=1cm];
	\draw (20:1) to[out=200,in=255] (75:1);
	\draw (40:1) to[out=220,in=-30] (150:1);
	
	\draw (-20:1) to[out=-200,in=-255] (-75:1);
	\draw (-40:1) to[out=-220,in=30] (-150:1);
	\draw (-125:1) to[out=-305,in=-270] (-90:1);
	
	\node at (0,-1.5) {$\sigma_4$};
\end{scope}
\end{tikzpicture}

\caption{Smoothings of $\sigma$.}
\label{fig:smoothings}
\end{figure}
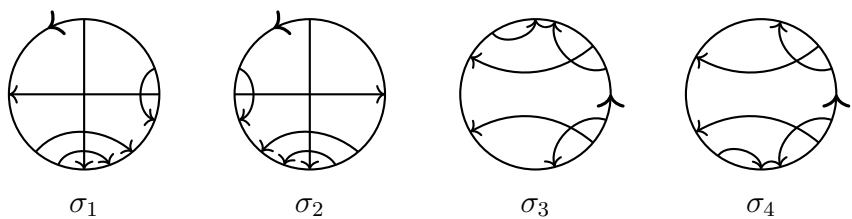

\begin{ex} Consider the $3$-string $\tau$ in Figure~\ref{fig:3string}. The multistring based matrix of $\tau$ is given in Figure~\ref{fig:tauMBM}. We claim that $T(\tau)=T_\bullet(\tau)$. Note that the self-arrows $g_2$ and $g_3$ are neither core nor annihilating and that this does not change under intersection moves. Hence it suffices to show that no intersection arrows can be removed by $M_4^{-1}$ after a sequence of intersection moves.

Consider the self-arrow $g_2$. No intersection moves may be applied to $g_2$ and so no pair of intersection arrows can become $g_2$-annihilating. Hence no pair of intersection arrows can be sum-annihilating, as desired.
\end{ex}

\begin{figure}
\begin{tikzpicture}[thick,->] 
	\draw[decoration={markings,mark=at position 0.5 with {\arrow[scale=1.5]{>}}},
        postaction={decorate}] (0,0) circle[radius=1cm];
	\draw (-45:1cm) to[out=-45,in=225] node[pos=0.5,below]{$x_2$} ([xshift=4cm] 225:1); 
	\draw ([xshift=4cm] 135:1) to[out=135,in=45] node[pos=0.5,above]{$x_1$} (45:1cm);
	
\begin{scope}[xshift=4cm] 
	\draw[decoration={markings,mark=at position .33 with {\arrow[scale=1.5]{>}}},
        postaction={decorate}] (0,0) circle[radius=1cm];
	\draw (180:1)-- node[auto]{$g_2$} (0:1); 
\end{scope}
	
\begin{scope}[xshift=8cm]	
	\draw[decoration={markings,mark=at position 1 with {\arrow[scale=1.5]{>}}},
        postaction={decorate}] (0,0) circle[radius=1cm];
	\draw (90:1)-- node[auto]{$g_3$} (270:1); 
	\draw (45:1) to[out=45,in=45] (-1.5,.75) to[out=225,in=-45] ([xshift=-4cm] -45:1); 
	\node[left] at (-1.5,.75) {$x_4$};
	\draw (225:1) to[out=225,in=270] node[pos=0.5,below]{$x_3$} ([xshift=-4cm] 270:1);
	\draw ([xshift=-4cm] 90:1) to[out=90,in=135] node[pos=0.5,above]{$x_6$} (135:1);
	\draw ([xshift=-4cm] 45:1) to[out=45,in=135] (-1.5,-.75) to[out=-45,in=-45] (-45:1);
	\node[left] at (-1.5,-.75) {$x_5$};
\end{scope}
	
\end{tikzpicture}
\caption{A 3-string $\tau$ with $T(\tau) = T_\bullet(\tau)$.}
\label{fig:3string}
\end{figure}
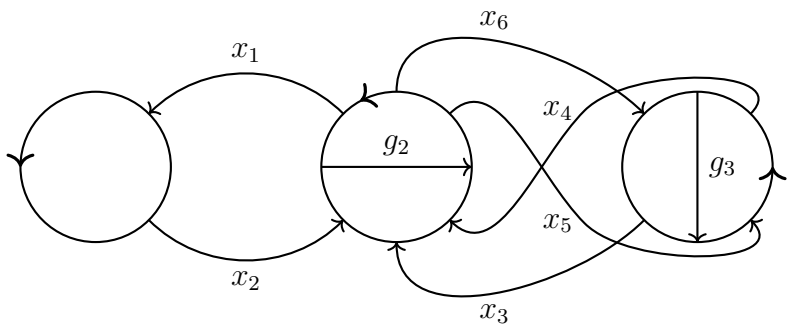

\begin{figure}
\begin{tabular}{c|ccccccccccc}
&$s_1$&$s_2$ & $g_2$ & $s_3$ & $g_3$ & $x_1$ & $x_2$ & $x_3$ & $x_4$ & $x_5$ & $x_6$\\\hline
$s_1$ & 0 & 0 & $1$ & 0 & 0 & $-1$ & $1$ & 0 & 0 & 0 & 0\\
$s_2$ & 0 & 0 & 0 & 0 & 0 & $1$ & $-1$ & $-1$ & $-1$ & $1$ & $1$\\
$g_2$ & $-1$ & 0 & 0 & $-2$ & $-1$ & 0 & $-1$ & $-1$ & $-1$ & 0 & 0\\
$s_3$ & 0 & 0 & $2$ & 0 & 0 & 0 & 0 & $1$ & $1$ & $-1$ & $-1$ \\
$g_3$ & 0 & 0 & $1$ & 0 & 0 & 0 & 0 & $1$ & 0 & 0 & $-1$ \\
$x_1$ & $1$ & $-1$ & 0 & 0 & 0 & 0 & 0 & 0 & 0 & 0 & 0\\
$x_2$ & $-1$ & $1$ & $1$ & 0 & 0 & 0 & 0 & 0 & 0 & 0 & 0\\
$x_3$ & 0 & $1$ & $1$& $-1$ & $-1$ & 0 & 0 & 0 & 0 & 0 & 0\\
$x_4$ & 0 & $1$ & $1$ & $-1$ & 0 & 0 & 0 & 0 & 0 & 0 & 0\\
$x_5$ & 0 & $-1$ & 0 & $1$ & 0 & 0 & 0 & 0 & 0 & 0 & 0\\
$x_6$ & 0 & $-1$ & 0 & $1$ & $1$ & 0 & 0 & 0 & 0 & 0 & 0\\
\end{tabular}
\caption{Multistring based matrix $T(\tau)$.}
\label{fig:tauMBM}
\end{figure}

\section{Invariants of Virtual Multistrings}

Analogous to Turaev's constructions of the $u$-polynomial and the $\rho$ invariant for virtual 1-strings (recall  Section \ref{sec:invariants}), we construct similar invariants for virtual homotopy classes of virtual multistrings by using multistring based matrices.

\subsection{$u$-invariant}

Let $\B$ be a virtual $n$-string and $T(\B) = (G_i,I,s_i,B)$ the associated multistring based matrix. For $g_i\in G_i$, define 
$n_j(g_i) = B(g_i,s_j)$ and let \[u_i([\B]) = \ds{\sum_{g_i\in G_i} \sign(n_i(g_i))t^{\abs{n_i(g_i)}}\prod_{j\neq i} \paren{x^{n_j(g_i)}+x^{n_j(s_i)-n_j(g_i)}}}\in \Z[t,x].\] Then the {\em $u$-invariant} of $[\B]$ is
the set $u([\B]) = \set{u_i([\B]):i=1,\dots,n}$. Note that this does not depend on the ordering of the components of $\B$. 

\begin{rem} We can obtain a polynomial from $u([\B])$ by adding the $u_i([\B])$'s together, although the resulting invariant is weaker because of potential canceling.
\end{rem} 

The next result shows that $u([\B])$ does not depend on the representative of $[\B]$ and hence defines an invariant for virtual multistrings.



\begin{prop}\label{prop:uinv} The $u$-invariant for $[\B]$ is well-defined.\end{prop}

\begin{proof}
By Proposition~\ref{prop:homotopy}, it suffices to show that applying elementary extensions and intersections moves to $T(\B)$ does not change the $u$-invariant.
\begin{itemize}[leftmargin=*]
\item Consider $g\in G_i$ which is $G_i$-annihilating or $G_i$-core. Then $\sign(B(g,s_i)) = 0$, and so $g$ contributes $0$ to $u_i([\B])$.

\item Consider a $G_i$-complementary pair $g_1,g_2$. Then
\[B(g_1,s_i) + B(g_2,s_i) = B(s_i,s_i) = 0 \Rightarrow \boxed{B(g_1,s_i) = - B(g_2,s_i)}(\ast).\]
More generally, for $j\neq i$, $n_j(s_i) = n_j(g_1) + n_j(g_2)$ and so
\[x^{n_j(g_1)} + x^{n_j(s_i)-n_j(g_1)} = x^{n_j(g_1)} + x^{n_j(g_2)} = x^{n_j(s_i)-n_j(g_2)} + x^{n_j(g_2)},\]
i.e., 
$\boxed{x^{n_j(g_1)} + x^{n_j(s_i)-n_j(g_1)} = x^{n_j(g_2)} + x^{n_j(s_i)-n_j(g_2)}}(\ast\ast)$. Finally, the contribution of $g_2$ to $u_i([\B])$ is the opposite of the contribution of $g_1$ by $(\ast)$ and $(\ast\ast)$.
Thus the contributions of $g_1$ and $g_2$ cancel, and so $G_i$-complementary pairs contribute 0 to $u_i([\B])$.

\item Consider a sum-annihilating pair $x_1,x_2\in I$. By definition, we have $B(g,x_1) + B(g,x_2)=0$ for all $g\in G$. Thus $x_1,x_2$ contribute 0 to $n_j(g)$ for all $g\in G$; hence 0 to $u([\B])$.
\item If we apply $I_1(g;x_1,x_2)$, then we still have $B(g,x_1)+B(g,x_2)=0$. 
\item If we apply $I_2(g;x_1,x_2)$, then we still have $B(g,x_1) \neq B(g,x_2)$ and $B(s_i,x_1)=B(s_i,x_2)$.
\end{itemize}
For the last two cases, $n_j(g)$ is unchanged and so $u([\B])$ is unchanged as well.
\end{proof}



\begin{prop} \label{prop:usum} Let $\B$ be a virtual $n$-string and $\A_i$ the virtual string associated to the $i$th core circle of $\B$. Then $u([\B])(t,1)$ is the $n$-tuple where each $u_i([\B])(t,1) = 2u([\A_i])(t)$.
\end{prop}


\begin{rem} If virtual $n$-string $\B$ is the disjoint union of $n$ virtual strings, then $u_i([\B])(t,x) = 2u([\A_i])(t)$. The converse, that the $u$-polynomial detects such split virtual multistrings, is false. Indeed, taking two copies of the virtual string with no arrows and adding a single intersection arrow produces a non-split 2-string $\B$ with $u([\B]) = \set{0,0}$. 
\end{rem}

%
	%
%
%
	%
%

\subsection{$\rho$ invariants}

Analogous to Turaev's $\rho$ invariant, we define a family of invariants using the uniqueness of primitive woven based matrices. 
Let $\B$ be a virtual $n$-string, $T_\bullet(\B) = (G_{i\bullet},I_\bullet,s_{i\bullet},B_\bullet)$ an associated primitive multistring based matrix, and let $G_\bullet = \ds{\bigcup_{i=1}^n G_{i\bullet}}$. Define
\begin{itemize}[leftmargin=*]
\item $\rho([\B]) = \#(G_\bullet\cup I_\bullet) - n$,
\item $\rho'([\B]) = \# G_{\bullet} -n$,
\item $\rho_\cap([\B]) = \# I_\bullet$.
\end{itemize}
Each quantity is an invariant because $T_\bullet(\B)$ is primitive. 
Moreover, we have the identity $\rho([\B]) = \rho'([\B]) + \rho_\cap([\B])$ since $G_\bullet\cap I_\bullet =\es$.

Geometrically, $\rho'([\B])$ is giving a lower bound on the minimal number of self-intersection arrows for $[\B]$; $\rho_\cap([\B])$ a lower bound on the minimal number of intersection arrows. 
If $T(\B)\cong T_\bullet(\B)$ is primitive, then the $\rho$-invariants are sharp estimates on their respective quantities.

For instance, we know that the family $\B(p_1,q_1,p_2,q_2,r,s)$ produces primitive multistring based matrices so long as the induced virtual strings do. In these cases, the diagram from Figure~\ref{fig:betafamily} realizes the minimal number of intersections and self-arrows.

If we fix an ordering on the core circles of $\B$, and require isomorphisms to preserve this order, then we can define individual $\rho$ invariants on the independent virtual strings $\A_i$. That is, let $\rho_i([\B]) = \# G_{i\bullet}-1$. Then, $\rho_i([\B])\geq \rho([\A_i])$ where we treat $[\A_i]$ as a virtual string and use Turaev's $\rho$ invariant. 

\subsection{Distinguishing virtual multistrings}

We conclude by showing that primitive multistring based matrices may be used to distinguish families of virtual strings.

\begin{prop}\label{prop:distinct} Let $\B_1,\B_2$ be 2-strings of the form $\B(p_1,q_1,p_2,q_2,r,s)$ where the induced virtual strings are nontrivial. Then $[\B_2]=[\B_1]$ if and only if $\B_2$ and $\B_1$ are equal as virtual 2-strings.
\end{prop}

\begin{proof}
Assume $\B_2$ and $\B_1$ are equal as virtual $2$-strings. Fix an order on the components of $\B_1$, say $(\A_1,\A_2)$. Then $\B_2=(\A_1,\A_2)$ or $\B_2=(\A_2,\A_1)$. In either case, as flat virtual links, $[\B_2]=[\B_1]$.

Now assume $\B_1$ and $\B_2$ are distinct virtual 2-strings. Since $\B_1$ and $\B_2$ are in the $\B$-family of 2-strings and the induced virtual strings are nontrivial, $T(\B_1)= T_\bullet(\B_1)$ and $T(\B_2)= T_\bullet(\B_2)$. Thus, to show that $[\B_2]\neq[\B_1]$, it suffices to prove that $T(\B_1)$ and $T(\B_2)$ are non-isomorphic.

Let $\B_1= \B(p_{11},q_{11},p_{21},q_{21},r_1,s_1)$ and $\B_2 = \B(p_{12},q_{12},p_{22},q_{22},r_2,s_2)$. Since $\B_1$ and $\B_2$ are distinct $2$-strings, we know that $\B_2\neq \B(p_{21},q_{21},p_{11},q_{11},s_1,r_1)$ and $\B_2\neq \B_1$. Since $\B_2\neq \B_1$, some pair of parameters are distinct:
\begin{itemize}[leftmargin=*]
\item $\boxed{p_{12}\neq p_{11}}$: Then $\B_1$ has $p_{11}$ arrows on component 1 with the weaving map evaluating to $-1$ on $s=\min\set{s_1,s_2}$ intersection arrows. However, $\B_2$ has $p_{12}$ such arrows.
\item $\boxed{r_2\neq r_1}$: Since the induced virtual strings are nontrivial, we know that $q_{11},q_{12}>0$. Then $\B_1$ has $q=\min\set{q_{11},q_{12}}$ self-arrows on component 1 which, when paired with $r_1$ intersection arrows, the weaving map evaluates to $+1$. However, on $\B_2$, $q$ self-arrows can be paired with $r_2$ arrows in this way.
\end{itemize}
The remaining 4 cases follow analogously. Thus there is no isomorphism of $T(\B_1)$ and $T(\B_2)$ preserving order of the components. By a similar argument, as $\B_2\neq \B(p_{21},q_{21},p_{11},q_{11},s_1,r_1)$, there is no isomorphism of $T(\B_1)$ and $T(\B_2)$ switching the order of the components. Hence $T(\B_1)$ and $T(\B_2)$ are non-isomorphic.
\end{proof}



{\small
\bibliographystyle{abbrv}
\bibliography{references}

\begin{thebibliography}{1}

\bibitem{Cahn}
P.~Cahn.
\newblock A generalization of {{Turaev}}'s virtual string cobracket and
  self-intersections of virtual strings.
\newblock {\em Commun. Contemp. Math.}, 19(4):1650053, 2017.

\bibitem{Carter}
J.~S. Carter, S.~Kamada, and M.~Saito.
\newblock Stable equivalence of knots on surfaces and virtual knot cobordisms.
\newblock {\em J. Knot Theory Ramifications}, 11(3):311--322, 2002.

\bibitem{Henrich}
A.~Henrich.
\newblock A sequence of degree one vassiliev invariants for virtual knots.
\newblock {\em J. Knot Theory Ramifications}, 19(4):461--487, 2010.

\bibitem{Manturov}
D.~P. Ilyutko, V.~O. Manturov, and I.~M. Nikonov.
\newblock Parity in knot theory and graph-links.
\newblock {\em J. Math. Sci.}, 193(6):809--965, 2013.

\bibitem{Kadokami}
T.~Kadokami.
\newblock Detecting {{Non}}-{{Triviality}} of {{Virtual Links}}.
\newblock {\em J. Knot Theory Ramifications}, 12(6):781--803, 2003.

\bibitem{Kauffman}
L.~H. Kauffman.
\newblock Virtual {{Knot Theory}}.
\newblock {\em European J. Combin.}, 20(7):663--691, 1999.

\bibitem{Turaev}
V.~Turaev.
\newblock Virtual strings.
\newblock {\em Ann. Inst. Fourier (Grenoble)}, 54(7):2455--2525, 2004.

\end{thebibliography}
}
\end{document}